\documentclass[12pt]{article}

\usepackage{amsmath,amsthm,amsfonts}
\usepackage{graphicx}
\usepackage{authblk}
\newtheorem{theorem}{Theorem}[section]
\newtheorem{proposition}[theorem]{Proposition}

\newtheorem{definition}[theorem]{Definition}

\numberwithin{equation}{section}

\usepackage{geometry}
\usepackage{xcolor}

\begin{document}

\title{Analytic proof of the emergence of new type of Lorenz-like attractors from the triple instability in systems with $\mathbb{Z}_4$-symmetry}

\author[1]{Efrosiniia Karatetskaia}
\author[1]{Alexey Kazakov}
\author[1]{Klim Safonov}
\author[2]{Dmitry Turaev}

\affil[1]{National Research University Higher School of Economics\\
25/12 Bolshaya Pecherskaya Ulitsa, 603155 Nizhny Novgorod, Russia}

\affil[2]{Imperial College, London SW7 2AZ, United Kingdom}

\date{eyukaratetskaya@gmail.com, kazakovdz@yandex.ru, \\ safonov.klim@yandex.ru, d.turaev@imperial.ac.uk}

\maketitle

\begin{abstract}
We study bifurcations of a symmetric equilibrium state in systems of differential equations invariant with respect to a $\mathbb{Z}_4$-symmetry. We prove that if the equilibrium state has a triple zero eigenvalue, then pseudohyperbolic attractors of different types can arise as a result of the bifurcation. The first type is the classical Lorenz attractor and the second type is the so-called Sim\'o angel. The normal form of the considered bifurcation also serves as the normal form of the bifurcation of a periodic orbit with multipliers $(-1, i, -i)$. Therefore, the results of this paper can also be used to prove the emergence of discrete pseudohyperbolic attractors as a result of this codimension-3 bifurcation, providing a theoretical confirmation for the numerically observed Lorenz-like attractors and Sim\'o angels in the three-dimensional H\'enon map.
\end{abstract}

\section{Introduction and main results}\label{sec1}

It was discovered by Arneodo, Coullet, Spiegel, and Tresser that the triple instability leads to the emergence of chaotic dynamics from an equilibrium state or a periodic orbit \cite{ACS83, ACT85, IR05}. In the general case, the chaotic dynamics generated by the triple instability are associated with the Shilnikov homoclinic loop \cite{Sh65} and are not robust due to the appearance of stability windows under arbitrarily small perturbations. However, Vladimirov and Volkov found \cite{VV93} that in the presence of an additional symmetry, the triple instability (three zero eigenvalues of the linearization matrix) can lead to the birth of the Lorenz attractor, i.e., to robust chaotic behavior. The Vladimirov-Volkov bifurcation creates complexified Lorenz attractors in systems with $O(2)$-symmetry. The triple instability in systems with an axial $\mathbb{Z}_2$-symmetry is described by the real version of the same normal form --- the Shimizu-Morioka system --- and leads to the birth of the classical (real) Lorenz attractor \cite{ASh86, S93, SST93}.

In the present paper, we study the appearance of robust chaotic attractors due to the triple instability in systems with $\mathbb{Z}_4$-symmetry. We derive the corresponding normal form --- a three-dimensional system of ordinary differential equations --- and prove that under certain conditions, in the parameter space, there are infinitely many open disjoint regions corresponding to the robust chaotic dynamics. We show that some of these regions correspond to the existence of the classical Lorenz attractor (like in the Afraimovich-Bykov-Shilnikov geometric model \cite{ABS77, ABS82}). In other regions, we discover robust chaotic attractors of a different nature. We call them Sim\'o angels (Fig.~\ref{Portrait_of_attractors}b,c), see the precise definition in Section \ref{Pseudohyperbolic_attractors}. A discrete analogue of the attractor shown in Fig.~\ref{Portrait_of_attractors}b was first observed by Carles Sim\'o when he performed numerical experiments with the 3D H\'enon map (compare Fig.~\ref{Portrait_of_attractors}b with Fig.~2b from \cite{GOST05}).

Let $\dot{w}=G_{\eta}(w)$ be a family of $n-$dimensional smooth vector fields depending smoothly on a set of parameters $\eta$. Assume that this family is invariant with respect to a $\mathbb{Z}_4$-symmetry $S$, i.e, $S$ is a diffeomorphism
such that $DS(w)\cdot G_{\eta}(w)=G_{\eta}\circ S(w)$ and $S^4={\rm Id}$. Assume that the vector field has a symmetric equilibrium state $O$ (i.e., $S(O)=O$) which coincides with the coordinates' origin. Let the equilibrium $O$ have three zero eigenvalues at $\eta=0$ and other eigenvalues do not lie on the imaginary axis. Then, for small values of $\eta$, there exists a smooth, locally invariant, 3-dimensional  center manifold $W^c_{loc}(O)$, which is also invariant with respect to the symmetry $S$.

We consider the restriction of the vector field to the invariant center manifold, therefore, from now on, we think that our system is defined in $\mathbb{R}^3$. For an attractor on the center manifold to be an attractor of the original system, it is necessary and sufficient that the linearization matrix at $O$ has no eigenvalues to the right of the imaginary axis.

By the Montgomery-Bochner Theorem \cite{BM46}, there exist coordinates $(x,y,z)$ such that the symmetry $S$ acts linearly on $W^c_{loc}(O)$. We assume that $S\neq {\rm Id}$, $S^2\neq {\rm Id}$ on the center manifold. Then, the symmetry is generated by one of the following matrices
$$
S=\begin{pmatrix}
0 & -1 & 0\\
1 & 0 & 0\\
0 & 0 & -1
\end{pmatrix} \ \ \ \text{or} \ \ \
S=\begin{pmatrix}
0 & -1 & 0\\
1 & 0 & 0\\
0 & 0 & 1
\end{pmatrix}.
$$
In this paper, we consider the first case, i.e., the matrix $S$ corresponds to the rotation by $\pi/2$ in the $(x,y)$-plane along with the flip along the $z$-axis. Introduce a complex variable  $u=x+i y$ and denote by $u^*=x-i y$ its conjugate. In the coordinates $(u,u^*,z)$, the vector field on the center manifold is invariant with respect to the linear symmetry defined by the matrix
\begin{equation}\label{eq:symmetry}
S=
\begin{pmatrix}
i & 0 & 0\\
0 & -i & 0\\
0 & 0 & -1
\end{pmatrix}.
\end{equation}
Therefore, it can be written as the following system of differential equations
\begin{equation}\label{FlowNormalFormGeneral}
\begin{aligned}
\dot{u}&=(-\gamma+i \beta )\, u+a_0 \, zu^*  + a_1\, u^2u^*+a_2\, z^2 u + a_3\, (u^*)^3+O_4, \\
\dot{z}&=\mu\, z+b_0\, u^2+b_0^*\, (u^*)^2 +b_1\, z^3+ b_2\, zuu^*+O_4,
\end{aligned}
\end{equation}
where $\gamma$, $\beta$, and $\mu$ vanish at the moment of bifurcation (three zero eigenvalues). We consider $\gamma$, $\beta$, and $\mu$ as small parameters that control the bifurcation unfolding. We assume that the coefficients $a_i, b_i$ are smooth functions of $\gamma$, $\beta$, $\mu$. Note that the coefficients $a_i, b_0$ take complex values, whereas the coefficients $b_1,b_2$, as well as the parameters $\gamma, \beta, \mu$, are real. Hereafter, $O_n$ denotes terms of order $n$ or higher; thus, $O_4$ in \eqref{FlowNormalFormGeneral} stands for terms of order $4$ and higher.

\begin{figure}[h]
    \centering
    \includegraphics[width=1.0\linewidth]{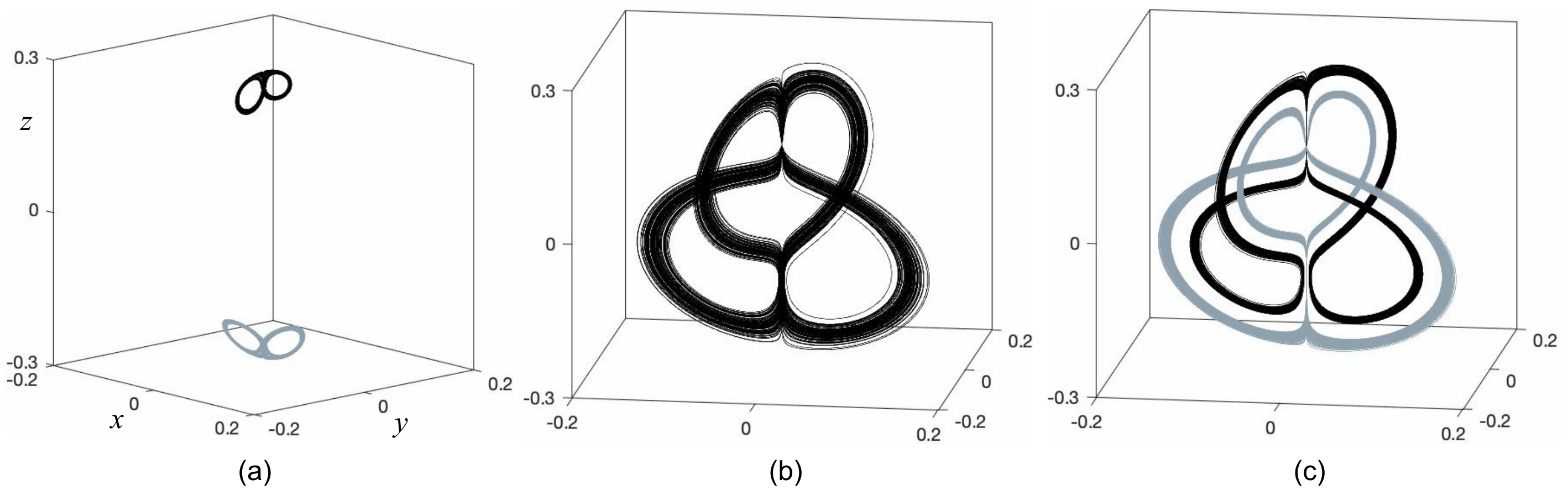}
    \caption{Phase portraits of pseudohyperbolic attractors of system \eqref{FlowNormalFormNumeric}, obtained numerically at $\mu = 0.02, \gamma = 0.07$: a) a symmetric pair of Lorenz attractors, $\beta = 0.16$; b) the four-winged Sim\'o angel, $\beta = 0.002$; c) a pair of two-winged Sim\'o angels, $\beta = 0.0015$.}
    \label{Portrait_of_attractors}
\end{figure}

In \cite{preprint} we study numerically a specific case of system \eqref{FlowNormalFormGeneral} given by
$$
\begin{aligned}
\dot{u}&=(-\gamma+i \beta )\, u-\frac{1-i}{2} \, zu^* +\frac{1-i}{8}\, u^2u^*+ \frac{3+i}{8}\, z^2 u - \frac{1-i}{8}\, (u^*)^3, \\
\dot{z}&=\mu\, z+\frac{1}{2}\, u^2+\frac{1}{2}\, (u^*)^2 -\frac{1}{4}\, z^3-\frac{1}{2}\, zuu^*,
\end{aligned}
$$
or, in  real variables,
\begin{equation}\label{FlowNormalFormNumeric}
\begin{aligned}
\dot{x}&=-\gamma\, x-\beta\, y-\frac{1}{2}\, z(x-y)+\frac{1}{2}\, xy(x+y)+z^2\left(\frac{3}{8}\, x-\frac{1}{8}\, y\right) , \\
\dot{y}&=\beta\, x-\gamma\, y+\frac{1}{2}\, z(x+y)+\frac{1}{2}\, xy(x-y)+z^2\left(\frac{1}{8}\, x+\frac{3}{8}\, y\right), \\
\dot{z}&=\mu \, z+x^2-y^2-\frac{1}{4}\, z^3-\frac{1}{2}\, z(x^2+y^2).
\end{aligned}
\end{equation}
It is found in \cite{preprint} that system \eqref{FlowNormalFormNumeric} has different types of chaotic attractors for small values of the parameters, see Figure~\ref{Portrait_of_attractors}.

In Fig. \ref{bifurcation_diagram} we show the diagram of the top Lyapunov exponent for system \eqref{FlowNormalFormNumeric} when $\mu=0.02$. The orange regions correspond to a positive Lyapunov exponent and, accordingly, to chaotic dynamics. The regions $LA$ and $SA$ in Fig. \ref{bifurcation_diagram} correspond to the existence of the Lorenz attractors and the Sim\'o angels, respectively.

\begin{figure}[h]
    \centering
    \includegraphics[width=1.0\linewidth]{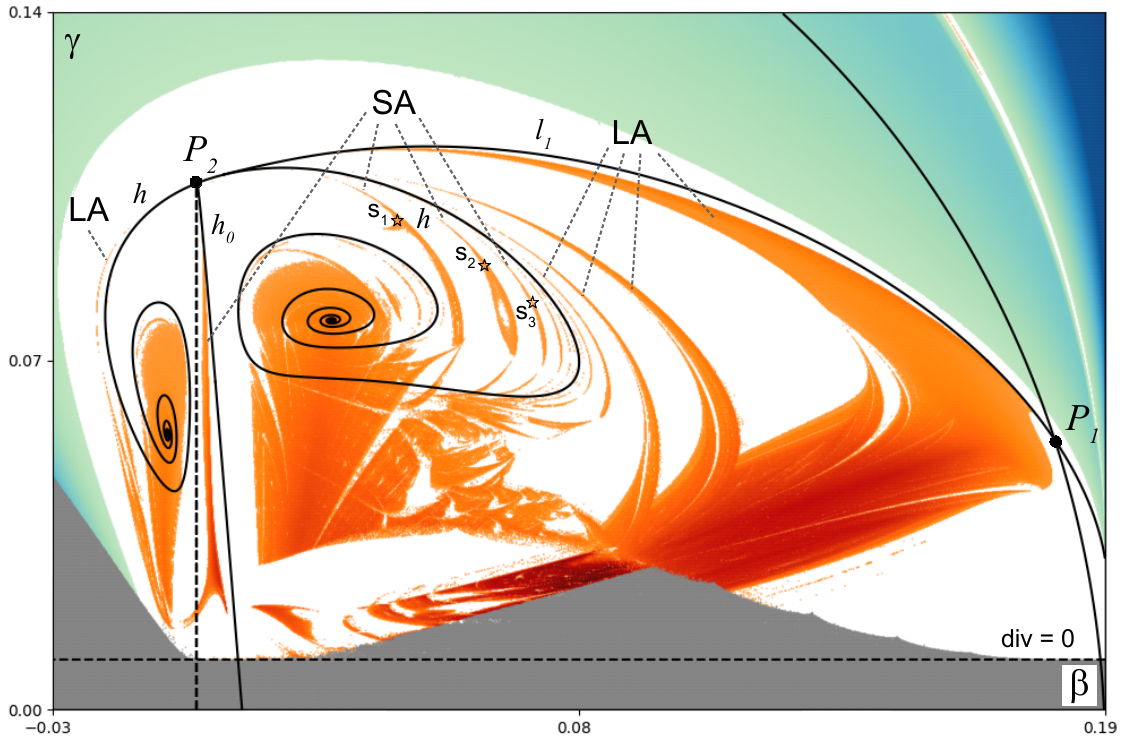}
    \caption{The diagram of the top Lyapunov exponent of system \eqref{FlowNormalFormNumeric} for $\mu=0.02$ and some important bifurcation curves. The regions $LA$ and $SA$ correspond to the existence of the Lorenz attractors and the Sim\'o angels, respectively. These regions start from the point $P_2$ where the curve $h$ corresponding to a heteroclinic connection shown in Fig.~\ref{HeteroclinicConnection} meets the line $\beta=0$ where the saddle $O(0,0,0)$ has two equal negative eigenvalues. The region LA also adjoins the point $P_1$ where the curve $l_1$ corresponding to a heteroclinic butterfly bifurcation with the saddle $O^-$ ($O^+$, by the symmetry) intersects with the curve where this saddle is resonant. The curves $l_2$ and $h$ correspond to the double homoclinic butterfly of the saddle $O^-$ ($O^+$) and the heteroclinic connection between $O^-$ and $O^+$.}
    \label{bifurcation_diagram}
\end{figure}

In this paper, we provide an analytic proof that such strange attractors indeed exist for open sets of values of $(\mu,\gamma,\beta)$ in system \eqref{FlowNormalFormNumeric}, as well as in general system \eqref{FlowNormalFormGeneral}, under appropriate conditions on the coefficients. First, we prove the following

\begin{theorem}\label{th1}
The bifurcation of a triple zero eigenvalue of a $\mathbb{Z}_4$-symmetric equilibrium such that the symmetry acts on the local center manifold as the multiplication by the matrix $S$ given by \eqref{eq:symmetry} leads to the birth of a pair of symmetric Lorenz attractors, if the coefficients of the normal form  \eqref{FlowNormalFormGeneral} satisfy the conditions
$$
b_1<0, \ \ {\rm Im}\, (a_0 b_0)\neq 0.
$$
\end{theorem}

The idea of the proof is as follows. We show that there exists a transformation of coordinates and time that brings system \eqref{FlowNormalFormGeneral} to the following form
\begin{equation}\label{SM_system1}
\begin{aligned}
\dot{x}&=y,  \\
\dot{y}&=x(1-z) -\lambda y+O(\sqrt{\mu}), \\
\dot{z}&=-\alpha z + x^2 +O(\sqrt{\mu}),
\end{aligned}
\end{equation}
where the coefficients $\lambda$ and $\alpha$ depend on the bifurcation parameters in such a way that they run all possible positive values when $\gamma, \beta$, and $\mu$ run near zero. When $\mu=0$, system \eqref{SM_system1} coincides with the Shimizu-Morioka system \cite{SM80}. The existence of a Lorenz attractor in the Shimizu-Morioka system for an open region in the $(\alpha,\lambda)$-plane was discovered numerically by Andrey Shilnikov in \cite{ASh86, S93}. A computer-assisted proof of the existence of the Lorenz attractor was given in \cite{CTZ18}. Since the Lorenz attractor is preserved by any $C^1$-small perturbation of a system (see Section \ref{Pseudohyperbolic_attractors}), it follows that system \eqref{SM_system1} also has the Lorenz attractor in some region of the parameter plane $(\alpha,\lambda)$ for all sufficiently small $\mu>0$. Therefore, system \eqref{FlowNormalFormGeneral} has the Lorenz attractor for an open region of small parameters $\gamma,\beta$, and $\mu$, as claimed.

In Fig. \ref{ShimizuMoriokaDiagram}a we show an enlarged fragment of the region $LA$ corresponding to the existence of the symmetric pair of Lorenz attractors. Note that this diagram is very similar to the diagram of the top Lyapunov exponent for the Shimizu-Morioka system, see Fig. \ref{ShimizuMoriokaDiagram}b.

\begin{figure}[h]
    \centering
    \includegraphics[width=1\linewidth]{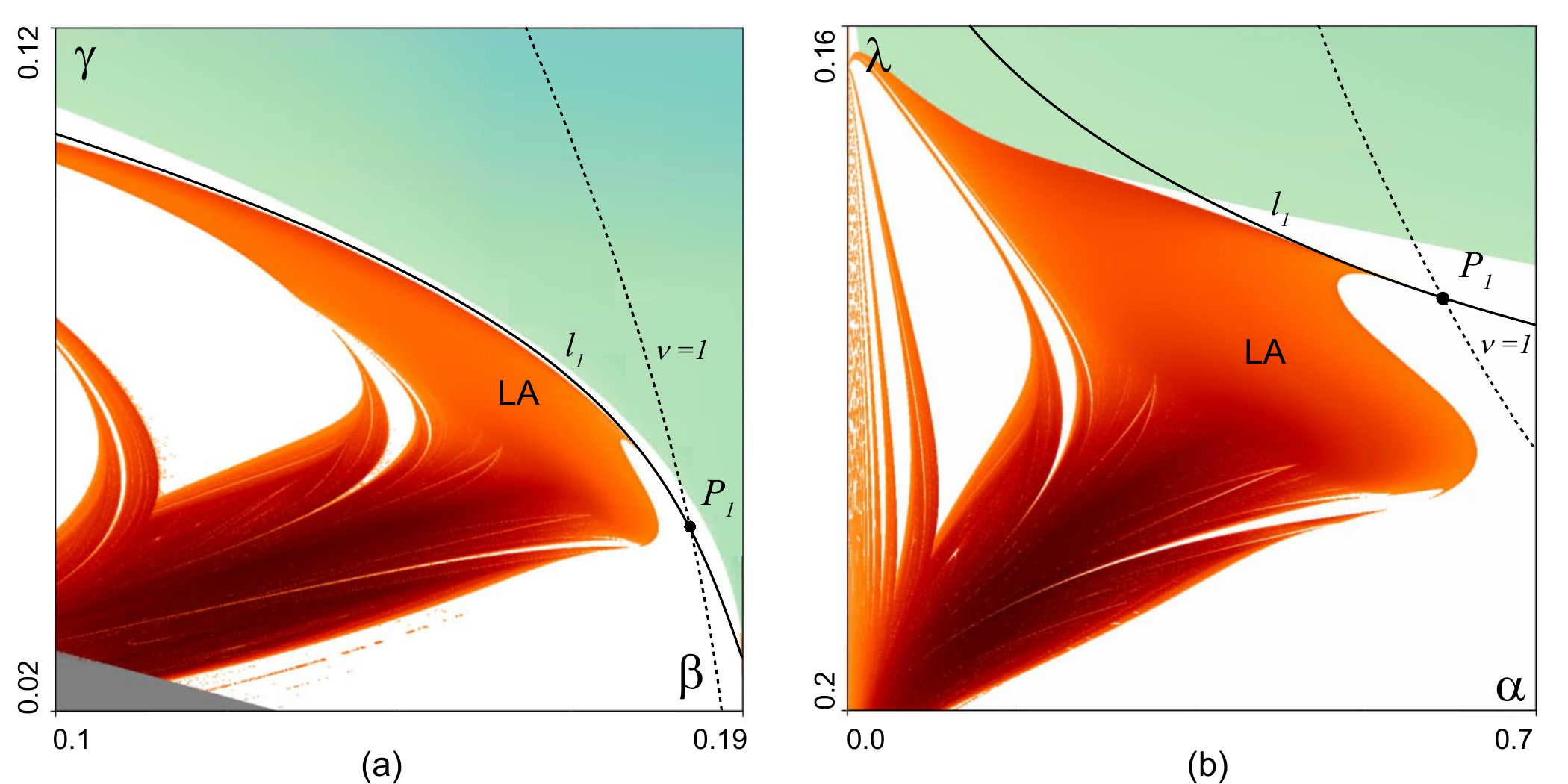}
    \caption{Lyapunov diagram a) for system \eqref{FlowNormalFormNumeric} and b) for the Simizhu-Morioka system. In both cases, the Lorenz attractor existence region LA adjoins the point $P_1$, which is the intersection of a homoclinic butterfly bifurcation curve $l_1$ and the curve $\nu=1$ corresponding to a resonant saddle equilibrium state (with the saddle index $\nu=1$).}
    \label{ShimizuMoriokaDiagram}
\end{figure}

Our next result describes the birth of strange attractors near the point $P_2$ in Fig. \ref{bifurcation_diagram}, where the regions $SA$ and $LA$ meet. This point is the intersection of the line $\beta=0$ and the bifurcation curve $h$ corresponding to the heteroclinic connection shown in Figure~\ref{HeteroclinicConnection}. System \eqref{FlowNormalFormNumeric} (as well as the general normal form \eqref{FlowNormalFormGeneral} for $b_1<0$ and $\mu>0$) has three equilibrium states on the $z$-axis. Note that the $z$-axis is invariant with respect to $S$. The two equilibria, which we denote by $O^-$ and $O^+$, are symmetric to each other with respect to $S$, i.e., $S(O^-)=O^+$ and $S(O^+)=O^-$. The equilibria $O^+$ and $O^-$ are saddles with one-dimensional unstable manifolds (each consisting of a pair of trajectories called unstable separatrices) and with two-dimensional stable manifolds (containing the semi-axes $z>0$ and $z<0$, respectively). The symmetric equilibrium $O$ at zero has a one-dimensional unstable manifold (the segment of the $z$-axis between $O^+$ and $O^-$) and a two-dimensional stable manifold.  If $\beta \neq 0$, then the equilibrium $O$ is a saddle-focus, i.e., the eigenvalues to the left of the imaginary axis are complex conjugate (see Fig.~\ref{HeteroclinicConnection}a). At $\beta=0$ the eigenvalues become real and equal to each other, so $O$ becomes a dicritical node on its stable manifold (see Fig.~\ref{HeteroclinicConnection}b). The bifurcation curve $h$ corresponds to the unstable manifolds of $O^+$, $O^-$ getting to the two-dimensional stable manifold of $O$, as given by the following

\begin{figure}[h]
    \centering
    \includegraphics[width=0.7\linewidth]{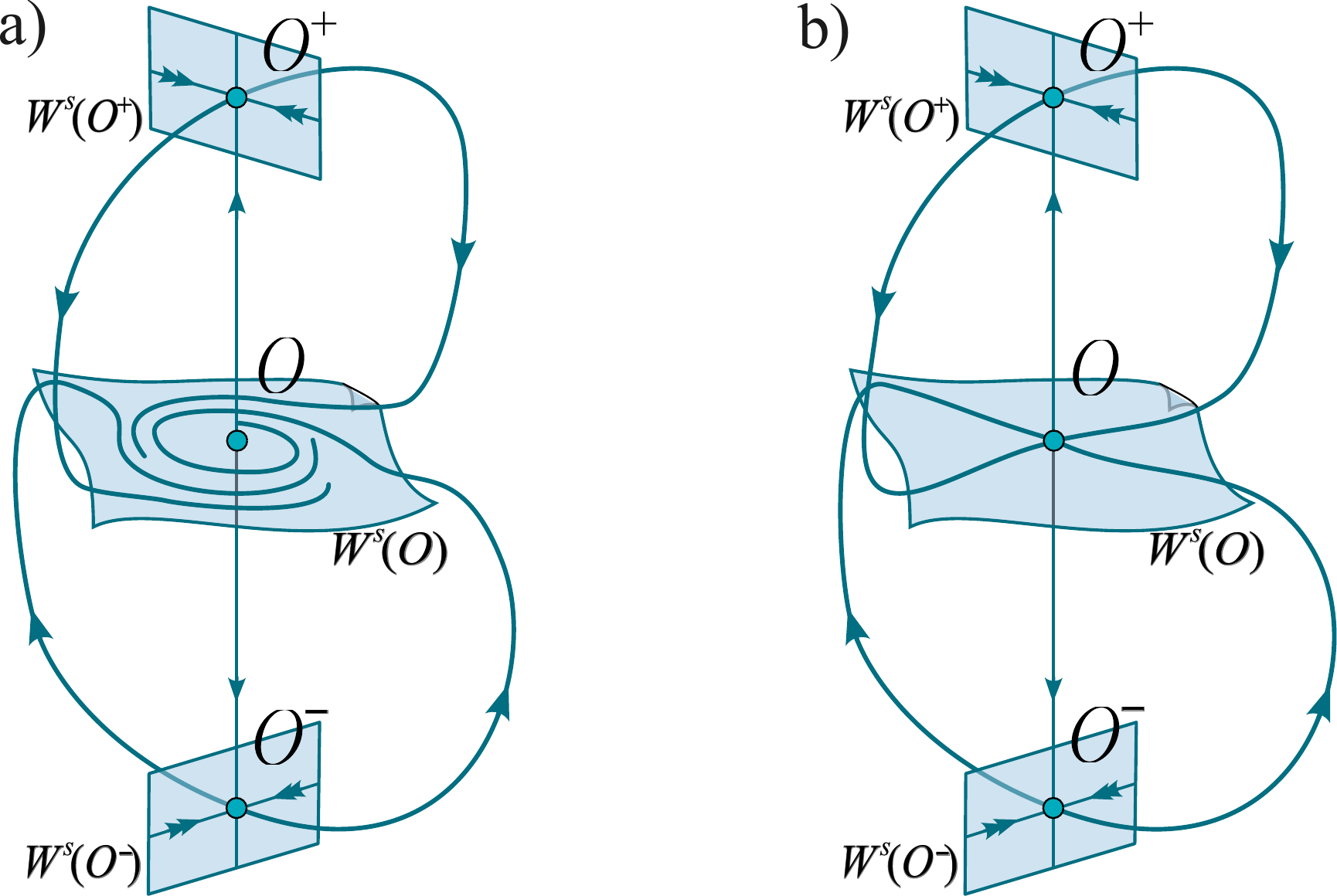}
    \caption{Four-winged heteroclinic connection for a) $\beta\neq 0$, b) $\beta=0$. The equilibrium $O$ is a saddle-focus at $\beta\neq 0$ and becomes a dicritical node on its stable manifold at $\beta=0$.}
    \label{HeteroclinicConnection}
\end{figure}

\begin{theorem}\label{th2}
Let the coefficients of the normal form  \eqref{FlowNormalFormGeneral} satisfy
$$
b_1<0, \ \ {\rm Re}\,  (a_0 b_0)<0.
$$
Then there exists a smooth function $h(\beta,\mu)$ defined for all sufficiently small $\beta$ and all sufficiently small $\mu\geq 0$ such that  $h(0,0)=0$ and system \eqref{FlowNormalFormGeneral} at $\gamma=h(\beta,\mu)$ has an $S$-symmetric four-winged heteroclinic connection such that the one-dimensional unstable separatrices of the equilibrium states $O^+$ and $O^-$ tend to the equilibrium state $O$ as $t\to+\infty$ (see Fig.~\ref{HeteroclinicConnection}).
\end{theorem}

We also prove that the regions $SA$ and $LA$ in Fig. \ref{bifurcation_diagram} indeed adjoin the point $P_2$.

\begin{theorem}\label{th3}
Let the coefficients of the normal form \eqref{FlowNormalFormGeneral} satisfy
$$
b_1<0, \ \ {\rm Re}\,  (a_0 b_0)<0.
$$
Then, for all sufficiently small $\mu>0$, in the half-plane $\gamma<h(\beta,\mu)$ of the $(\gamma,\beta)$-plane there exists a countable sequence of open disjoint regions $\mathcal {A}^-_k$ which adjoin the point $(\gamma,\beta)=(h(0,\mu),0)$ and correspond to the existence of a symmetric pair of the Lorenz attractors. Similarly, in the half-plane $\gamma>h(\beta,\mu)$ there exists a countable sequence of open disjoint regions $\mathcal {A}^+_k$ which adjoin the point $(\gamma,\beta)=(h(0,\mu),0)$ and correspond to the existence of a Sim\'o angel.
\end{theorem}


We derive this theorem from Theorem \ref{th2} using Theorem \ref{th6} in Section \ref{Shilnikov_criterion}. In \cite{Sh81} L.P. Shilnikov proposed three cases of homoclinic butterfly bifurcations with additional degeneracy at an equilibrium state leading to the birth of the Lorenz attractor (see also \cite{Rob89,Rob92,Rych90,GKKK22}). Theorem \ref{th6} gives another criterion in the same spirit for the birth of Lorenz-like attractors. In our case, we have the figure-eight heteroclinic connection instead of a homoclinic butterfly and an additional degeneracy (multiple stable eigenvalues) at the equilibrium $O$. In Section \ref{Shilnikov_criterion} we study this heteroclinic bifurcation in a general setting (not related to system \eqref{FlowNormalFormGeneral}).

\newpage

\section{Pseudohyperbolic attractors}\label{Pseudohyperbolic_attractors}

In this section, we define what we mean by the Lorenz and Sim\'o attractors. We base our definitions on the geometric model of the Lorenz attractor by Afraimovich, Bykov, and Shilnikov \cite{ABS77,ABS82} and the notion of pseudohyperbolicity \cite{TS98,TS08}.

Pseudohyperbolicity is a key property that ensures the robustness of chaotic attractors. It is a weak version of hyperbolicity. For the purposes of this paper, it is enough to consider the notion of pseudohyperbolicity which is more restrictive than general \cite{CTZ18, GKT21}. Namely, we characterize pseudohyperbolic attractors by the existence of a splitting of the tangent bundle in a neighborhood of an attractor into a direct sum of two invariant continuous subbundles $E^{ss}$ and $E^{cu}$. The linearized system uniformly contracts any vector in $E^{ss}$, while in $E^{cu}$ it expands $k$-dimensional volumes, where $k$ is the dimension of $E^{cu}$. Pseudohyperbolicity allows the linearized system to contract some vectors in $E^{cu}$ (as opposed to hyperbolicity), but it requires that any possible contraction in $E^{cu}$ must be uniformly weaker than any contraction in $ E^{ss}$. The expansion of volumes in $E^{cu}$ implies that the top Lyapunov exponent is positive for every orbit of the attractor.

\begin{definition} A vector field $V$ in $\mathbb{R}^n$ is pseudohyperbolic in a strictly forward invariant domain $\mathcal{D} \subset \mathbb{R}^n$ if:
\begin{itemize}
\item[$(1)$] at each point $w\in\mathcal{D}$ there exists a pair of transversal subspaces $E^{cu}(w)$ and $E^{ss}(w)$,
depending continuously on the point, such that the families of these subspaces are invariant with respect to the derivative $DV_t$ of the time-$t$ map $V_t$ of the system, i.e., $$
DV_t\, E^{cu}(w) = E^{cu}(V_t(w)), \ \ \ DV_t\, E^{ss}(w) = E^{ss}(V_t(w)) \text{ for all } t \geq 0;
$$
\item[$(2)$] there exist constants $C_{1,2}> 0$, $\sigma_{1,2}> 0$ such that for each $w \in \mathcal{D}$ and all $t \geq 0$
\begin{equation}\label{eq:pseudo_1}
\| DV_t(w)|_{E^{ss}} \| < C_1e^{-\sigma_1 t},
\end{equation}
\begin{equation}\label{eq:pseudo_2}
\|DV_t(w)|_{E^{ss}}\| \cdot \| DV_t(w)|_{E^{cu}}\|^{-1} < C_2e^{-\sigma_2 t};
\end{equation}
\item[$(3)$] there exist constants $C_3 > 0$ and $\sigma_3> 0$ such that for each $w \in \mathcal D$ and all $t \geq 0$
\begin{equation}\label{eq:pseudo_3}
{\rm det}(DV_t(w)|_{E^{cu}}) >  C_3e^{\sigma_3 t}.
\end{equation}
\end{itemize}
\end{definition}

It follows from inequality \eqref{eq:pseudo_2} that the decomposition $E^{cu} \times E^{ss}$ is continuously preserved by small $C^1$-perturbations of the vector field and, by the continuity, the decomposition keeps satisfying conditions (2) and (3). Thus, the property of pseudohyperbolicity is $C^1$-persistent. This means that, robustly with respect to $C^1$-small perturbations, every orbit of any attractor in $\mathcal D$ has a positive top  Lyapunov exponent.

Following \cite{TS98}, we define an attractor as a completely stable chain-transitive invariant set. Namely, denote by $w(t;w_0)$ a trajectory of a vector field passing through a point $w_0$ at $t=0$. A set $\{w_0,w_1,\ldots, w_n\}$ is called an {\it $(\varepsilon,\tau)$-trajectory} if there exists a set of times $0=t_0<t_1<\ldots<t_n$ such that $|t_i-t_{i-1}|>\tau>0$ and $\|w_{i}-w(t_{i-1};w_{i-1})\|<\varepsilon$ for $i=1,\ldots n$. A closed invariant set $A$ is called {\it chain-transitive} if for any points $w_1, w_2\in A$ and for any $\varepsilon>0$, $\tau>0$, the set $A$ contains an $(\varepsilon,\tau)$-trajectory connecting $X_1$ and $X_2$. A compact invariant set $A$ is {\it completely stable} if for an open neighborhood $U$ of $A$ there exist constants $\varepsilon, \tau>0$ and a smaller neighborhood $\tilde{U}: A\subset \tilde{U}\subset U$ such that an $(\varepsilon,\tau)$-trajectory starting at $\tilde{U}$ does not leave $U$.

We now proceed to describe our geometric model. Let $\dot{w}=G(w)$ be a $3$-dimensional smooth vector field invariant with respect to the symmetry $S$. Assume that
\begin{itemize}
    \item[(A1)] there are two symmetric saddle equilibrium states $O^-$, $O^+$ lying on the $z$-axis;
    \item[(A2)] the eigenvalues $\lambda_1, \lambda_2$, and $\lambda_3$ of the equilibrium states $O^-$, $O^+$ satisfy
    $$
    \lambda_1>0>\lambda_2>\lambda_3, \ \ \ \lambda_1+\lambda_2>0,
    $$
    and the eigenspace corresponding to $\lambda_2$ coincides with the $z$-axis.
\end{itemize}
Assumptions A1 and A2 imply that the pseudohyperbolicity conditions are fulfilled at the points $O^-$ and $O^+$, where $E^{cu}$ is the two-dimensional eigenspace corresponding to the eigenvalues
$\lambda_1$ and $\lambda_2$ and $E^{ss}$ is the eigenspace corresponding to the eigenvalue $\lambda_3$. This allows us to assume that
\begin{itemize}
\item[(A3)] the vector field is pseudohyperbolic  with ${\rm dim}(E^{cu})=2$ and ${\rm dim}(E^{ss})=1$ in a bounded open domain $\mathcal{D}$ containing $O^-$ and $O^+$;
\item[(A4)] there are two symmetric two-dimensional cross-sections $\Pi^-$ and $\Pi^+$ which lie in $\mathcal D$ and are transversal to the stable manifolds of $O^-$ and $O^+$, respectively, such that the forward trajectory of any point in $\mathcal{D}\backslash (W^s(O^-)\cup W^s(O^+))$ transversely intersects $\Pi^-\cup \Pi^+$, see Figure~\ref{GeometricModels}.
\end{itemize}

Assumptions A3 and A4 imply that there are at most two attractors in $\mathcal D$. To explain this, recall that a point $w$ is said to be {\it attainable} from a point $w_0$ if there exists $\tau>0$ such that for any $\varepsilon>0$ there exists an $(\varepsilon,\tau)$-trajectory connecting the points $w_0$ and $w$. Denote by $\Lambda(O^-)$ and $\Lambda(O^+)$ the set of all points attainable from $O^-$ and $O^+$, respectively. The sets $\Lambda(O^-)$ and $\Lambda(O^+)$ are chain-transitive closed invariant sets. These sets can intersect and then, due to chain transitivity, they coincide. By similar arguments as in \cite{TS98,TS08}, it follows from A3 and A4 that the union of the stable manifolds $W^s(O^-)\cup W^s(O^+)$ is dense in $\mathcal D$. This means that from any point $w\in \mathcal D$ one can attain $O^-$ or $O^+$, and the sets $\Lambda(O^-)$ and $\Lambda(O^+)$ are the only possible chain-transitive completely stable invariant sets in $\mathcal D$.

Assumption A4 implies that the dynamics of the vector field are determined by the Poincar\'e map $T$ of $\Pi^-\cup \Pi^+$. Here we can have two different situations. Namely, denote by $\Pi^{-}_{1}$ and $\Pi^{-}_{2}$ (or $\Pi^{+}_{1}$ and $\Pi^{+}_{2}$) the two halves of $\Pi^-\backslash W^s(O^{-})$  (resp., $\Pi^+\backslash W^s(O^{+})$). By assumption A4, every forward trajectory of a point in $\Pi^-_1\cup\Pi^-_2$ intersects $\Pi^+\cup \Pi^-$ and the intersection point depends continuously on the initial point. Thus, all trajectories starting at $\Pi^-_1$ (or, by the symmetry, at $\Pi^-_2$)  either return to $\Pi^-$ and do not intersect $\Pi^+$, as in Fig.~\ref{GeometricModels}a, or they intersect $\Pi^+$ first, as in Fig.~\ref{GeometricModels}b.

In the first case, we say that the system satisfying assumptions A1-A4 has {\it a symmetric pair of Lorenz attractors}. Here, any forward trajectory starting at $\Pi^-_1\cup \Pi^-_2$  returns to $\Pi^-$ and does not intersect $\Pi^+$, see Fig.~\ref{GeometricModels}a. By the symmetry, the same holds for the cross-section $\Pi^+$.
The Lorenz attractors are the two disjoint sets $\Lambda(O^-)$ and $\Lambda(O^+)$ of all points attainable from the saddle equilibria $O^-$ and $O^+$, respectively. The absorbing domain $\mathcal{D}$ consists of two connected components --- the absorbing domains of the Lorenz attractors $\Lambda(O^-)$ and $\Lambda(O^+)$. 

In the second case, any forward trajectory starting at $\Pi^-_1\cup \Pi^-_2$ intersects $\Pi^+$ and, by the symmetry, any forward trajectory starting at $\Pi^+_1\cup \Pi^+_2$ intersects $\Pi^-$, see Fig. \ref{GeometricModels}b; the absorbing domain $\mathcal D$ is connected here. 
In this case, we say that the system has {\it a Sim\'o angel attractor}. It is possible in this case that $\Lambda(O^-)$ and $\Lambda(O^+)$ are disjoint, then these sets comprise a pair of two-winged Sim\'o attractors, as in Fig.\ref{Portrait_of_attractors}c. It is also possible for $\Lambda(O^-)$ and $\Lambda(O^+)$ to coincide, then the set $\Lambda(O^-)=\Lambda(O^+)$ is the four-winged Sim\'o angel, as in Fig.\ref{Portrait_of_attractors}b. To study the Sim\'o angels, it is also convenient to consider the half-map $\hat{T}=S\circ T$, which maps $\Pi^-$ and $\Pi^+$ into itself.

\begin{figure}[h]
    \centering
    \includegraphics[width=0.7\linewidth]{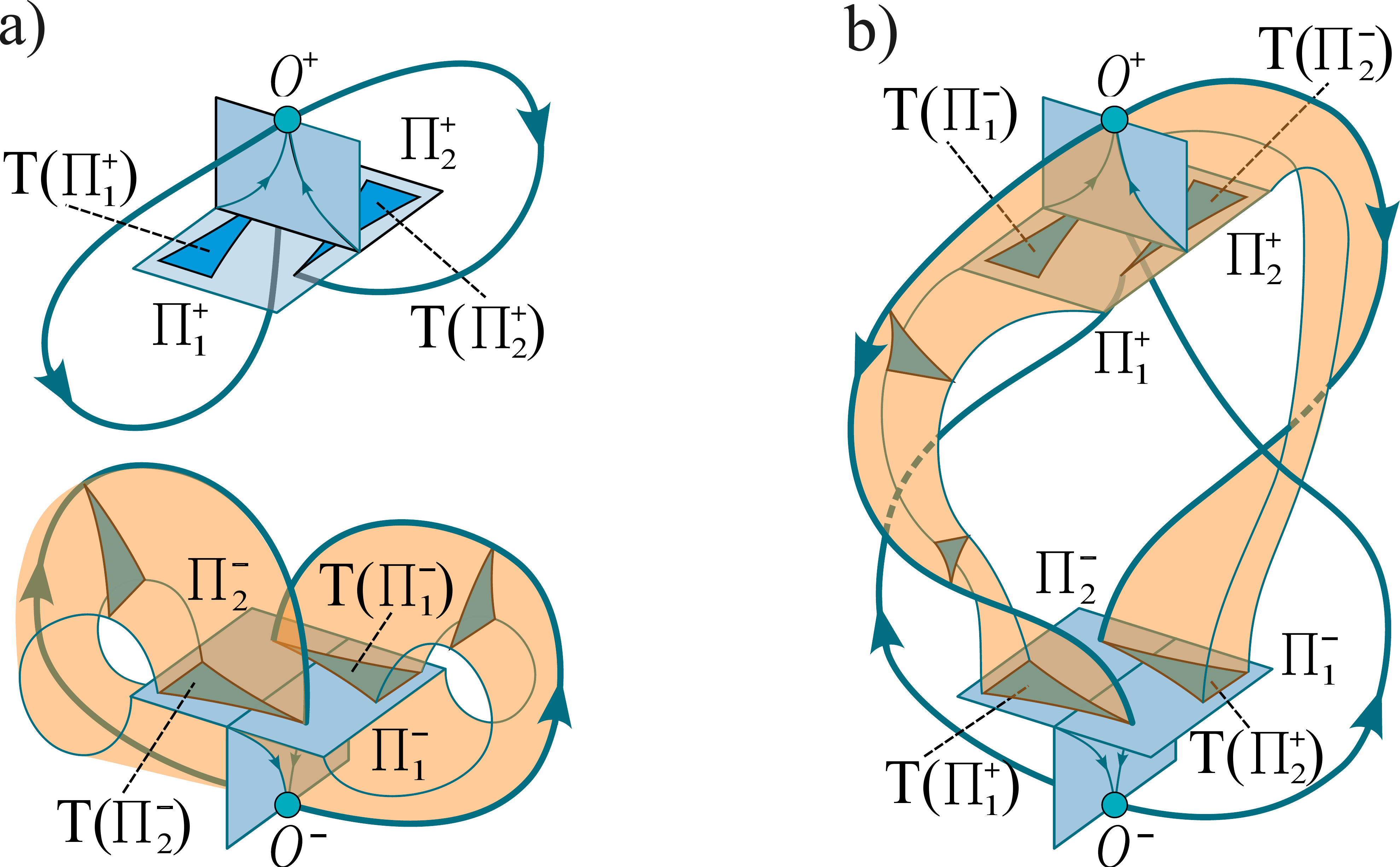}
    \caption{Geometry of the Poincar\'e map $T$ a) in the case of the Lorenz attractors, and b) in the case of the Sim\'o angel.}
    \label{GeometricModels}
\end{figure}

The assumption of pseudohyperbolicity for these attractors implies the hyperbolicity of the corresponding Poincar\'e maps (and vice versa, under assumption A2).

\begin{definition}\label{def:SingularHuyperbolicity}
The map $T$ is singular hyperbolic if:
\begin{itemize}
\item[$(a)$] there are a pair of continuous cone fields $C^{ss}$ and $C^{u}$ on $\Pi^-$ and $\Pi^+$ which complement each other, and the tangents to $W^s(O^-)$ and $W^{s}(O^+)$ lie in the interior of $C^{ss}$; \\
\item[$(b)$] the cones are closed, have a non-empty interior, and are invariant with respect to the derivative $DT$ of $T$:
$$
DT (C^{u}_P)\subset {\rm int}\, C^{u}_{T(P)}, \ \ \ C^{ss}_{T(P)}\subset {\rm int}\, DT(C^{ss}_P) \ \ \text{for} \ \ P\notin W^{s}(O^-)\cup W^{s}(O^+);
$$
\item[$(c)$] $DT$ uniformly contracts all vectors of $C^{ss}$ and uniformly expands all vectors of $C^u$.
     \end{itemize}
\end{definition}

Note that the cone conditions of this definition can be expressed as certain inequalities on the elements of the derivative matrix $DT$ of the Poincar\'e map, as it was done in the Afraimovich-Bykov-Shilnikov geometric model \cite{ABS77,ABS82}.

In Section \ref{Shilnikov_criterion}, we study bifurcations of the four-winged heteroclinic connection shown in Fig. \ref{HeteroclinicConnection}b. For each small $\mu$, this bifurcation corresponds to a point in the $(\alpha,\beta)$-parameter plane (the point $P_2$ in Fig. \ref{bifurcation_diagram}). We call this point a \textit{stem point} because, as we show, infinitely many regions corresponding to the existence of the Lorenz and Sim\'o attractors adjoin this point. Namely, we prove that for parameter values from these regions, one can choose two symmetric cross-sections satisfying assumption A4, and we derive explicit formulas for the Poincar\'e map. We show (see Proposition \ref{th:Poincare_map}) that the map $T$ for the Lorenz attractor and the map $\hat{T}$ for the Sim\'o angel are given by the following equations
\begin{equation}\label{eq:Poincare_map_1}
\begin{aligned}
\bar{X}&=(1+\varepsilon |Y|^{\nu}\,  \phi_1(X,Y))\,{\rm sign}(Y),\\
\bar{Y}&=\sigma \, (-1+|Y|^{\nu}\,(c+ \varepsilon\, \phi_2(X,Y)) )\, {\rm sign}(Y),
\end{aligned}
\end{equation}
where $\sigma$ is equal to $1$ or $-1$, the functions $\phi_{1,2}$ are smooth and bounded, and the coefficient $\varepsilon$ tends to $0$ as the parameters approach the stem point, the coefficient $c$ is a function of the bifurcation parameters whose values run a certain interval for which map \eqref{eq:Poincare_map_1} has the Lorenz attractor. The coefficient $\nu=-\lambda_2/\lambda_1$ is called the saddle index of the equilibria $O^{\pm}$; in system \eqref{FlowNormalFormGeneral}, $\nu$ is $O(\sqrt\mu)$ close to $0$. In the coordinates $(X,Y)$, the cross-section $\Pi=\Pi^-$ (or $\Pi^+$) is a square $\max(|X|,|Y|)<1+O(\varepsilon)$. The curve $W^s_{loc}(O^{\pm})\cap \Pi$ corresponds to the line $Y=0$. It is the discontinuity line of map \eqref{eq:Poincare_map_1}. When $Y\to \pm 0$, the image $(\bar{X},\bar{Y})$ tends to the points $M_1=(1,-\sigma)$ or $M_2=(-1,\sigma)$, where the unstable separatrices intersect $\Pi$.

For small $\varepsilon$ the map is strongly contracting in $X$. 
For $\nu<1$ it is strongly expanding in $Y$ for small $Y$, but may be contracting in $Y$ for $Y$ close to $1$. We study maps of type \eqref{eq:Poincare_map_1} in the upcoming paper \cite{preprint2}. We show there that for any $\nu>0$ there exists an interval of $c$ values for which the attractor of the corresponding map is singular hyperbolic in the sense of Definition~\ref{def:SingularHuyperbolicity}, i.e., there exists a transformation of the coordinate $Y$ such that the map becomes uniformly expanding in $Y$ and, thus, satisfies the Afraimovich-Bykov-Shilnikov conditions. In this way, we show that the Lorenz and Sim\'o attractors are born from the stem point.

As we mentioned earlier, small $\mu$ in the normal form \eqref{FlowNormalFormNumeric} corresponds to small $\nu$ in map~\eqref{eq:Poincare_map_1}. We show in \cite{preprint2} that for small $\nu$ the attractor of map \eqref{eq:Poincare_map_1} is transitive (i.e., there exists a trajectory that is dense in the attractor) and topologically mixing. It is the closure of the set of hyperbolic periodic points and contains an $\omega$-limit set of every point in $\Pi$.

\begin{figure}[h]
    \centering
    \includegraphics[width=0.4\linewidth]{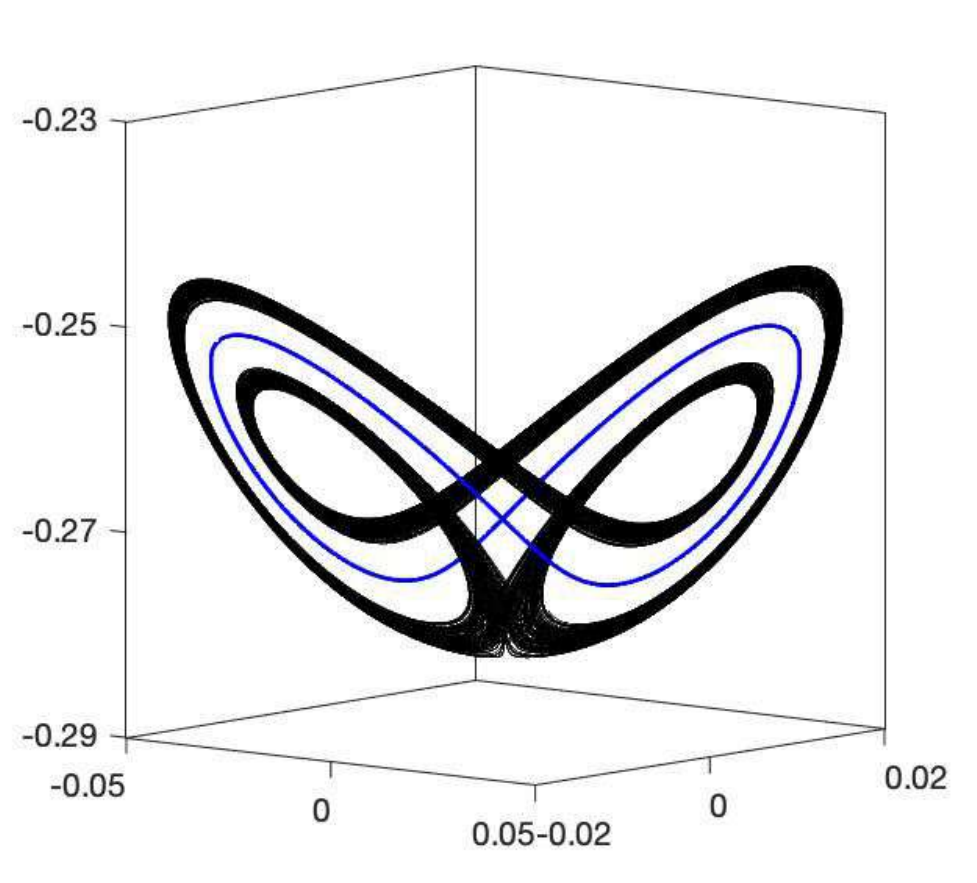}
    \caption{The Lorenz attractor with a trivial lacuna in system \eqref{FlowNormalFormNumeric} for $(\beta, \gamma, \mu) = (0.1775, 0.0475, 0.02)$. The lacuna contains a saddle periodic trajectory.}
    \label{Lorenz_with_lacunae}
\end{figure}

\begin{figure}[h]
    \centering
    \includegraphics[width=0.9\linewidth]{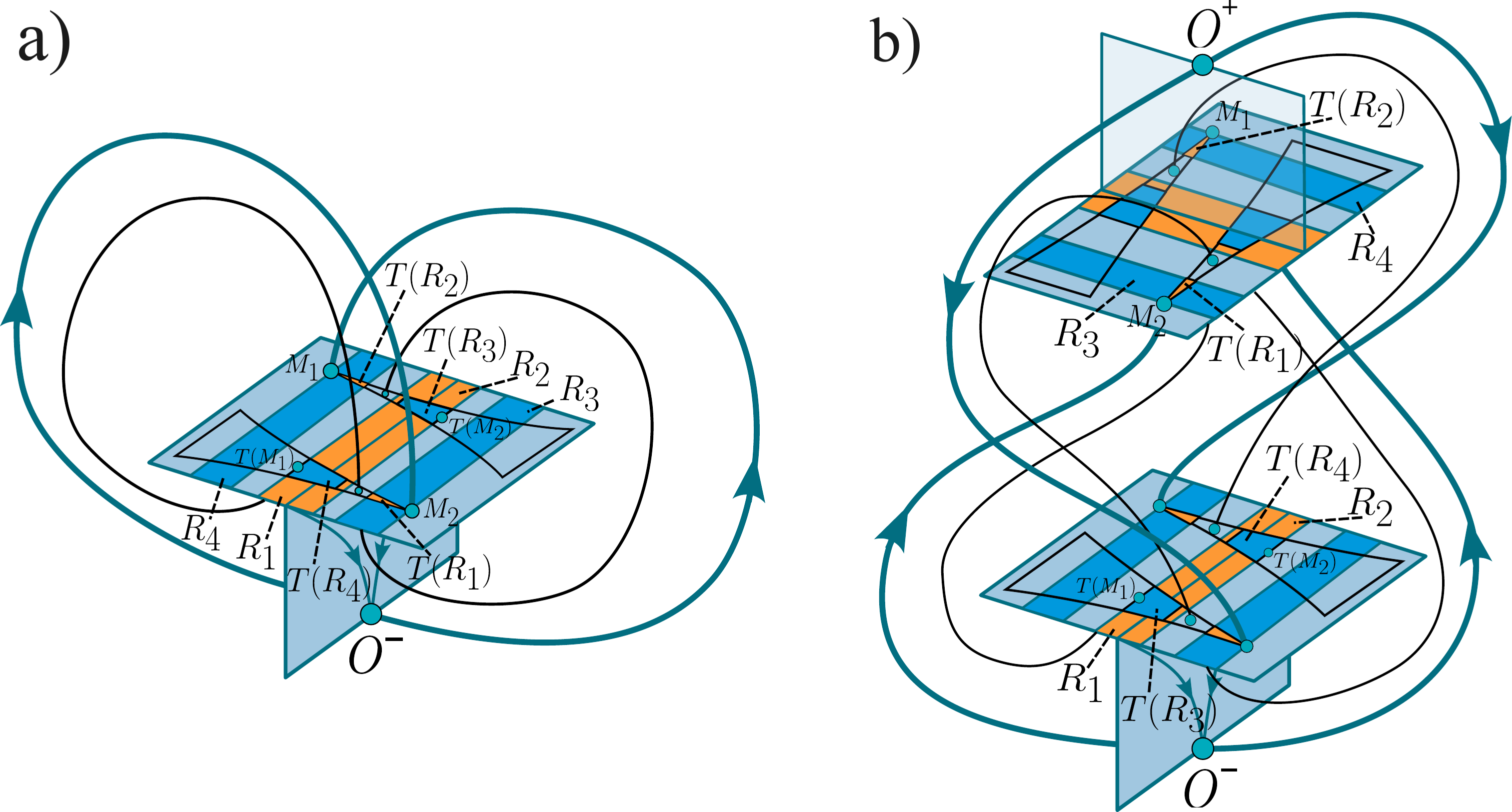}
    \caption{The Poincar\'e map a) for the Lorenz attractor with a trivial lacuna and b) for a pair of two-winged Sim\'o attractors. The points $M_1$ and $M_2$ correspond to the intersection of the unstable separatrices of $O^-$ with $\Pi^-$ or $\Pi^+$. In both cases, the rectangles $R_1$ and $R_2$ adjoin $W_{loc}^s(O^-)\cap \Pi^-$ and contain in their boundaries the points $T(M_1)$ and $T(M_2)$, respectively. The images $T(R_1)$ and $T(R_2)$ are contained in $R_3$ and $R_4$, respectively, so that $M_2$ and $T^2(M_1)$ lie in the boundary of $R_3$, whereas  $M_1$ and $T^2(M_2)$ lie in the boundary of $R_4$. The images $T(R_1)$ and $T(R_2)$ are contained in $R_3 \cup R_4$ and intersect $W_{loc}^s(O^-)\cup \Pi^-$, the points $T^3(M_1)$, $T^3(M_2)$ lie strictly inside $R_3 \cup R_4$.}
    \label{LorenzLacuna}
\end{figure}

It follows from the theory of \cite{preprint2} that for small $\mu$ near the stem point, the attractor on the cross-section consists of two connected components, one containing the point $M_1$ and the other containing the point $M_2$. However, with an increase of $\mu$ or when moving away from the stem point, the so-called lacuna may emerge in the attractor, see Fig. \ref{Lorenz_with_lacunae}. In the case of the symmetric pair of Lorenz attractors, this happens when a figure-eight saddle periodic orbit departs from each of the attractors. In this case, the dynamics of the Poincar\'e map can be described as follows, see Fig.~\ref{LorenzLacuna}a. In the cross-section $\Pi$ there are four rectangles: two of them, $R_1$ and $R_2$, adjoin the line $W^s_{loc}(O^{\pm}) \cap \Pi$ and the other two, $R_3$ and $R_4$, contain the points $M_1$ and $M_2$. The images of $R_1$ and $R_2$ are contained in $R_3$ and $R_4$, respectively. Each of the images $R_3$ and $R_4$ is contained in $R_1\cup R_2$ and intersects $W^s_{loc}(O^{\pm})$. So, the intersection of the attractor with $\Pi$ consists of four connected components, one in $R_3$, one in $R_4$, and two in $R_1\cup R_2$. Each of the attractors $\Lambda(O^-)$ and $\Lambda(O^+)$ is transitive and asymptotically stable. Every orbit, except for those in the stable manifold of the figure-eight periodic orbit, tend to the attractor.

In the case of the Sim\'o attractor, the creation of the lacuna is shown in Fig.~\ref{LorenzLacuna}b. Here, a four-round saddle periodic orbit departs from the four-winged Sim\'o attractor, and the attractor splits into a pair of two-winged Sim\'o attractors (containing $O^+$ and $O^-$, respectively). The dynamics of the half-map $\hat{T}$ are the same as for the map $T$ in the Lorenz case. Let us describe the dynamics of the Poincar\'e map $T$ in more detail. In this case, we can choose two rectangles $R_1, R_2\subset \Pi^-$ adjoining $W^s_{loc}(O^-) \cap \Pi^-$ and two rectangles $R_3, R_4\subset \Pi^+$ containing the points of the first intersection of the unstable separatrices of $O^-$ with $\Pi^+$, see Fig.\ref{LorenzLacuna}b. As in the Lorenz attractor case, the images of $R_1$ and $R_2$ are contained in $R_3$ and $R_4$, respectively. Each of the images $R_3$ and $R_4$ is contained in $R_1\cup R_2$ and intersects $W^s_{loc}(O^-)$. Each of the four rectangles contain four connected components of the intersection of the attractor $\Lambda(O^-)$ with $\Pi^-\cup \Pi^+$, and these rectangles, as well as the attractor $\Lambda(O^-)$, are separated from $W^s_{loc}(O^+)$. By the symmetry, there are also four rectangles which contain four connected components of the intersection of the other attractor $\Lambda(O^+)$ with $\Pi^-\cup \Pi^+$ and are separated from $W^s_{loc}(O^-)$. Both $\Lambda(O^-)$ and $\Lambda(O^+)$ are transitive and asymptotically stable and every orbit that does not tend to the four-round periodic orbit in the lacuna tends to one of these attractors.

In the general case, the Lorenz and Sim\'o attractors may have more complicated lacunae and more connected components of intersection with the cross-sections. It follows from the Lorenz attractor theory \cite{ABS77,ABS82} that lacunae contain a compact invariant hyperbolic set $\Sigma$, which is trivial (has finitely saddle periodic orbits) or equivalent to a suspension of a non-trivial finite Markov chain. When the non-trivial lacuna appears, the attractor may become non-transitive -- this happens when $\Sigma$ collides with the transitive component of the attractor. We have not seen non-trivial lacunae in system~\eqref{FlowNormalFormNumeric}, however they must exist here because, as we show (in Section~\ref{sec:Existence_of_LA}) the system is close to the Shimizu-Morioka model where non-trivial lacunae were found in \cite{Bob24}.

\section{Existence of the Lorenz attractor}\label{sec:Existence_of_LA}

Let us give conditions under which system~\eqref{FlowNormalFormNumeric} has a pair of symmetric Lorenz attractors. Suppose that system \eqref{FlowNormalFormGeneral} satisfies the non-degeneracy condition $a_0, b_0, b_1\neq 0$ and, moreover, we focus on the case
\begin{equation}\label{eq:coefficients_condition}
b_1<0, \ \ {\rm Re}  (a_0 b_0)<0, \ \ {\rm Im} (a_0 b_0)\neq 0.
\end{equation}
Due to the symmetry, the $z$-axis is an invariant line for system \eqref{FlowNormalFormGeneral}. The system restricted to this line has the form
$$\dot{z}=\mu\, z+b_1\, z^3+O(z^4).$$
When $\mu$ changes sign, a pitchfork bifurcation takes place on the $z$-axis. The sign of the coefficient $b_1$ determines the type of this bifurcation. For $b_1<0$ the equilibrium $O(0,0,0)$ undergoes the supercritical pitchfork bifurcation: for $\mu<0$ the equilibrium is stable on the $z$-axis,  and for $\mu>0$ it becomes unstable and two stable equilibria $O^+(0,0,z_0)$ and $O^-(0,0,-z_0)$ are born, where
$$
z_0=\sqrt{\frac{\mu}{|b_1|}}+O(\mu).
$$
Assuming $\mu>0$, we rescale the coordinates and time:
\begin{equation}\label{mu_scaling}
u=z_0 \sqrt[4]{|b_1|}\, \exp\left(\frac{i \, {\rm Arg}(a_0)}{2}\right) \, u_{new}, \;\;\; z=z_0 \, z_{new},\;\;\; t_{new}=z_0|a_0|\; t.
\end{equation}
After that we obtain the system
\begin{equation}\label{FlowNormalFormRescaled}
\begin{aligned}
\dot{u}&=(-\rho+i \omega )\, u+  z u^* +\sqrt{\mu}\, f_1(u,u^*,z) , \\
\dot{z}&=(A+iB)\, u^2+(A-iB)\, (u^*)^2 +\sqrt{\mu} \, f_2(u,u^*,z),
\end{aligned}
\end{equation}
where
\begin{equation}\label{eq:high_order_terms}
\begin{aligned}
f_1(u,u^*,z)&=\frac{a_1}{|a_0|}\,  u^2u^*+ \frac{a_2}{ |a_0|\, \sqrt{|b_1|}} \, z^2u+\frac{a_3\,|a_0|}{a_0^2} \, (u^*)^3+\sqrt{\mu} \, O_3,\\
f_2(u,u^*,z)&=\left(\frac{\sqrt{b_1}}{|a_0|}+O(\sqrt{\mu})\right)\, (z-z^3) +\frac{b_2}{|a_0|}\, z u u^*+\sqrt{\mu}\, O_3.
\end{aligned}
\end{equation}
The $O_3$-terms in \eqref{eq:high_order_terms} are such that system \eqref{FlowNormalFormRescaled} is $S$-symmetric; in particular, $f_1$ vanishes at $u=u^*=0$. Also, $f_2$ and $\partial f_2/\partial (u,u^*)$ vanish at the equilibrium states $(u,u^*,z)=(0,0,\pm 1)$ and $(u,u^*,z)=(0,0,0)$.

The new parameters $\rho$ and $\omega$ are equal to
\begin{equation}\label{eq:rho_omega_definition}
\rho=\frac{\gamma}{z_0  |a_0|}, \ \ \ \omega=\frac{\beta}{z_0 |a_0|},
\end{equation}
and are not necessarily small. We also have
$$
A+i B= a_0 b_0 \, \frac{\sqrt{|b_1|}}{|a_0|^2}.
$$
Note that by \eqref{eq:coefficients_condition}
$$
A<0, \ \ B\neq 0.
$$
Below, we study bifurcations of system \eqref{FlowNormalFormRescaled} in the parameter space $(\rho,\omega,\mu)$. Note that we rescaled the coordinates in such a way that in the new coordinates the equilibrium states lying on the $z$-axis do not move when the parameters change and are fixed $O(0,0,0)$, $O^+(0,0,1)$, $O^-(0,0,-1)$.

In this section, we prove that, for some open region of parameters $(\rho, \omega, \mu)$, system \eqref{FlowNormalFormRescaled} has two symmetric Lorenz attractors, one contains the equilibrium states $O^+$ and the other contains the equilibrium $O^-$. To do this, we show that for some parameter values, system \eqref{FlowNormalFormRescaled} is close to the Shimizu-Morioka model in a neighborhood of the equilibrium $O^+$ and, by the symmetry, in a neighborhood of the equilibrium $O^-$. The Shimizu-Morioka model is a system given by the following differential equations
\begin{equation}\label{SM_system}
\begin{aligned}
\dot{x}&=y,  \\
\dot{y}&=x(1-z) -\lambda y, \\
\dot{z}&=-\alpha z + x^2,
\end{aligned}
\end{equation}
where $\lambda$ and $\alpha$ are positive parameters.

\begin{theorem}\label{ThreeZeros}
Let $B\neq 0$ in system \eqref{FlowNormalFormRescaled}. Take any compact region $\mathbb K$ in the $(\alpha,\lambda)$-plane corresponding to positive $\alpha, \lambda$ and any ball $\mathbb B \subset \mathbb{R}^3$. Then, there exist constants $K_{1,2,3}$ such that for $\rho$ and $\omega$ satisfying
\begin{equation}\label{eq:omega_rho_equations}
\begin{aligned}
|\rho|\leq K_1 \sqrt{\mu}, \ \ \ \
K_2 \mu\leq \omega-{\rm sign}(B)-\frac{{\rm Im}\,(a_2)}{|a_0|\sqrt{|b_1|}}\sqrt{\mu}\leq K_3 \mu, \\
\end{aligned}
\end{equation}
there exists a transformation of coordinates, parameters, and time, smooth for $\mu>0$, which brings system \eqref{FlowNormalFormRescaled} to the form
\begin{equation}\label{eq:SM_plus_small_terms}
\begin{aligned}
\dot{x}&=y,  \\
\dot{y}&=x(1-z) -\lambda y+O(\sqrt{\mu}), \\
\dot{z}&=-\alpha z + x^2 +O(\sqrt{\mu}).
\end{aligned}
\end{equation}
Here, when $\mu$ is sufficiently small, the new coordinates $(x,y,z)$ run the ball $\mathbb B$ and the range of the new parameters $\alpha$ and $\lambda$ covers the region $\mathbb{K}$.
\end{theorem}

\begin{proof}
It was shown in \cite{SST93} that the Shimizu-Morioka system is a normal form for a certain local bifurcation of an equilibrium state with triple zero eigenvalues. To prove the theorem we will find values of the parameters when the equilibrium $O^-$ and $O^+$ undergo a degenerate version of such bifurcation. Then, by performing similar calculations as in \cite{SST93}, we bring system \eqref{FlowNormalFormRescaled} to form \eqref{eq:SM_plus_small_terms}.

In system \eqref{FlowNormalFormRescaled}, shift the origin to the point $O^-(0,0,-1)$ and obtain the system
\begin{equation}\label{eq:shifted_system}
\begin{aligned}
\dot{u}&=(-\rho+C\sqrt{\mu}+i (\omega +D\sqrt{\mu})+O(\mu))\, u+(1+O(\mu))\, u^*+ u^* z + \sqrt{\mu}\, O_2, \\
\dot{z}&=\left(-\frac{2\sqrt{b_1}}{|a_0|}\, \sqrt{\mu}+O(\mu)\right)\, z +(A+iB)\, u^2+(A-iB)\, (u^*)^2  +\sqrt{\mu}\, O_2,
\end{aligned}
\end{equation}
where we denote
\begin{equation}\label{eq:CD_def}
C+iD= \frac{a_2}{|a_0|\, \sqrt{|b_1|}}.
\end{equation}
The constants $O(\mu)$ come from the $\sqrt{\mu}$-terms in the functions $f_{1,2}$ of \eqref{eq:high_order_terms}. Note that system \eqref{eq:shifted_system} is symmetric with respect to $u\to -u$. Thus, the quadratic terms in the first equation of \eqref{eq:shifted_system} vanish at $z=0$.

Rewrite system \eqref{eq:shifted_system} in real variables $(x,y,z)$, where $u=x+i y$:
\begin{equation}\label{FlowNormalFormRescaledO+}
\begin{aligned}
\dot{x}&=(1-\rho+C\sqrt{\mu}+O(\mu) )\, x - (\omega+D\sqrt{\mu}+O(\mu)) \, y+ zx+\sqrt{\mu}\, (z\, O_1+O_3),  \\
\dot{y}&=(\omega+D\sqrt{\mu}+O(\mu))\, x + (-1-\rho+C\sqrt{\mu}+O(\mu)) \, y -zy+\sqrt{\mu}\, (z\, O_1+O_3), \\
\dot{z}&=\left(-\frac{2\sqrt{b_1}}{|a_0|}\, \sqrt{\mu}+O(\mu)\right)\, z + 2A\,(x^2-y^2)-4B\, xy+\sqrt{\mu}\, O_2.
\end{aligned}
\end{equation}
The linearization matrix of this system at the origin equals to
$$
\begin{pmatrix}
1-\rho+C\sqrt{\mu}+O(\mu) & -\omega-D\sqrt{\mu}-O(\mu) & 0 \\
\omega+D\sqrt{\mu}+O(\mu) & -1-\rho+C\sqrt{\mu}+O(\mu) & 0 \\
0 & 0 & -\frac{2\sqrt{b_1}}{|a_0|}\, \sqrt{\mu}+O(\mu)
\end{pmatrix}.
$$
For $\mu=\rho=0$ and $\omega=\pm 1$, all three eigenvalues of this matrix are zero.  We consider the situation where $\mu$ is small and, according to \eqref{eq:omega_rho_equations}, $\rho=O(\sqrt{\mu})$ and $\omega={\rm sign}(B)+O(\sqrt{\mu})$. Introduce a new variable
$$
y_{new}=\dot{x}=(1-\rho+C\sqrt{\mu}+O(\mu))\, x - (\omega+D\sqrt{\mu}+O(\mu)) \, y+ zx+\sqrt{\mu}\, (z\, O_1+O_3).
$$
Then system \eqref{FlowNormalFormRescaledO+} takes the form
\begin{equation}\label{FlowNormalFormRescaledO+2}
\begin{aligned}
\dot{x}&=y,  \\
\dot{y}&=\mu_1\,  x +\mu_2\, y+ 2zx+O_3+\sqrt{\mu}\,  z\, O_1, \\
\dot{z}&=\left(-\frac{2\sqrt{b_1}}{|a_0|}\, \sqrt{\mu}+O(\mu)\right)\, z -4|B|\, x^2 -2A\, y^2+4(A+|B|)\, xy+O_3+\sqrt{\mu}\, O_2,
\end{aligned}
\end{equation}
where we have two new small parameters $\mu_1$ and $\mu_2$ given by the formulas
\begin{equation}\label{eq:mu12}
\begin{aligned}
\mu_1&=
\begin{vmatrix}
1-\rho+C\sqrt{\mu}+O(\mu) & \omega+D\sqrt{\mu}+O(\mu) \\
\omega+D\sqrt{\mu}+O(\mu) & 1+\rho-C\sqrt{\mu}-O(\mu) \\
\end{vmatrix}=1-(\omega+D\sqrt{\mu})^2-(\rho-C\sqrt{\mu})^2+O(\mu), \\
\mu_2&=-2\rho+2C\sqrt{\mu}+O(\mu)= -2\rho+O(\sqrt{\mu}).
\end{aligned}
\end{equation}
Next, we put
$$
\lambda=-\frac{\mu_2}{\sqrt{\mu_1}}, \ \  \alpha=\frac{2\sqrt{b_1}}{|a_0|}\,\frac{\sqrt{\mu}}{\sqrt{\mu_1}}.
$$
It is easy to see from \eqref{eq:mu12} that by choosing the constants $K_1, K_2, K_3$ in \eqref{eq:omega_rho_equations} we can make $\alpha,\lambda$ run over any given region $\mathcal{D}$ if $\mu$ is small enough.

We complete the proof by noting that the following rescaling of coordinates and time brings system \eqref{FlowNormalFormRescaledO+2} to the required form \eqref{eq:SM_plus_small_terms}:
\begin{equation}\label{last_scaling}
x=\frac{\mu_1^{3/4}}{\sqrt{8|B|}}\, x_{new}, \ \ y=\frac{\mu_1^{5/4}}{\sqrt{8|B|}}\, y_{new}, \ \ z=-\frac{\mu_1}{2}\,  z_{new}, \ \ t_{new}=\sqrt{\mu_1}\, t.
\end{equation}

\end{proof}

As we mentioned above, since the Shimizu-Morioka system possesses the Lorenz attractor for an open region of the parameters $\alpha$, $\lambda$ \cite{S93,CTZ18}, and the Lorenz attractor is preserved by any $C^1$-small perturbation of a system \cite{ABS77,ABS82}, we obtain that system \eqref{eq:SM_plus_small_terms} has the Lorenz attractor for some region of parameters $\gamma,\beta$ for all sufficiently small $\mu>0$.

According to \eqref{eq:rho_omega_definition} and \eqref{eq:omega_rho_equations}, the Lorenz attractor exists in system \eqref{FlowNormalFormGeneral} for all small $\mu>0$ and some region of $\gamma,\beta$ satisfying
$$
\gamma =O(\mu), \ \ \ \beta =\sqrt{\mu}\,  \frac{|a_0|}{\sqrt{|b_1|}}\, \left({\rm sign}({\rm Im}(a_0 b_0)) -\frac{{\rm Im}(a_2)}{|a_0|\sqrt{|b_1|}}\, \sqrt{\mu}+O(\mu)\right).
$$

{\bf Remark 1.} {\it In the case $b_1>0$, in system \eqref{FlowNormalFormGeneral} the equilibrium states $O^+$ and $O^-$ are born for $\mu<0$. For parameter values $(\rho,\omega,\mu)$ close to the point $(0,{\rm sign}(B),0)$, applying the same arguments as in the proof of Theorem \ref{ThreeZeros}, we can transform  system \eqref{FlowNormalFormGeneral} to the form
$$
\begin{aligned}
\dot{x}&=y,  \\
\dot{y}&=x(1-z) +\lambda y+O(\sqrt{|\mu|}), \\
\dot{z}&=\alpha z - x^2 +O(\sqrt{|\mu|}),
\end{aligned}
$$
where $\alpha$ is positive. After the transformation $t\to -t, \ \ y\to -y$, this system takes form \eqref{eq:SM_plus_small_terms}. Hence, a pair of Lorenz repellers is born in this case.}

{\bf Remark 2.} {\it The normal form \eqref{FlowNormalFormGeneral} also serves as the normal form for the bifurcation of a fixed point of a diffeomorphism with multipliers $(-1,i,-i)$. It means that there are coordinates in a neighborhood of the fixed point such that the diffeomorphism is close to the composition of the symmetry $S$ and the time-one map generated by system \eqref{FlowNormalFormGeneral}. In particular, the H\'enon map
\begin{equation}\label{HenonMap}
    \bar{x}=y,\;\;\;
    \bar{y}=z,\;\;\;
    \bar{z}=M_1+Bx+M_2y-z^2.
\end{equation}
has the point $\left(\frac{1}{2},\frac{1}{ 2},\frac{1}{2}\right)$ with  multipliers $(-1,i,-i)$ at $M_1=\frac{7}{4}$, $M_2=-1$, $B=-1$, and system \eqref{FlowNormalFormNumeric} is the normal form for the corresponding bifurcation.

Theorem \ref{ThreeZeros} implies the emergence of the discrete Lorenz attractors (under certain conditions on the coefficients of the normal form) as a result of the bifurcation of a fixed point with multipliers $(-1,i,-i)$. In particular, this gives an analytic proof of the existence of the discrete Lorenz attractors near such bifurcation in the 3D H\'enon map \eqref{HenonMap}, see \cite{preprint}.}

\section{Four-winged heteroclinic connection}\label{sec4}

In this section, we show that if $A<0$ in system \eqref{FlowNormalFormRescaled}, then there exists a bifurcation surface $\rho=h(\mu,\omega)$ corresponding to the four-winged heteroclinic connection such as shown in Fig.~\ref{HeteroclinicConnection}a.

\begin{theorem}\label{th5}
Let $A<0$ in system \eqref{FlowNormalFormRescaled}. Then in the parameter space $(\rho,\omega,\mu>0)$ there exists a bifurcation surface
\begin{equation}\label{eq:het_curve}
\rho=1/2+k_1\,\sqrt{\mu}+k_2\, \omega+O(\mu+\omega^2),
\end{equation}
where
$$
k_1=\frac{\sqrt{b_1}}{12|a_0|}+ \left({\rm Re}(a_1)+\frac{{\rm Im}(a_1) B}{3A} \right)\,\frac{1}{12|A|\,|a_0|}+\left({\rm Re}(a_2)-\frac{{\rm Im}(a_2) B}{6 A}\right)\, \frac{1}{3|a_0|\,\sqrt{|b_1|}}+
$$
$$
+\left( {\rm Re}(a_3/a_0^2)+\frac{{\rm Im}(a_3/a_0^2)B}{3A}\right)\, \frac{|a_0|}{12|A|}+\frac{b_2}{24|A||a_0|},
$$
$$
k_2=\frac{(\pi^2-9)B}{3A}.
$$
such that system \eqref{FlowNormalFormRescaled}, \eqref{eq:high_order_terms} has an $S$-symmetric four-winged heteroclinic connection such that the one-dimensional unstable separatrices of the equilibrium states $O^+$ and $O^-$ tend to the equilibrium state $O$ as $t\to+\infty$, see Fig. \ref{HeteroclinicConnection}. As $\rho$ changes, the heteroclinic connections split with a non-zero velocity.
\end{theorem}

\begin{proof}
Consider system \eqref{FlowNormalFormRescaled}, \eqref{eq:high_order_terms} at $\omega=\mu=0$ and write it in real variables:
\begin{equation}\label{eq:system_rho_mu_zero}
\begin{aligned}
\dot{x}&=-\rho\, x + zx,  \\
\dot{y}&= - \rho \, y -zy, \\
\dot{z}&=2A\, (x^2-y^2)-4B\, xy.
\end{aligned}
\end{equation}
The coordinate planes $x=0$ and $y=0$ are invariant with respect to this system. The restriction of system \eqref{eq:system_rho_mu_zero} to the plane $y=0$ is
\begin{equation}\label{eq:system_rho_mu_zero_restriction}
\begin{aligned}
\dot{x}&=-\rho\,  x + zx,  \\
\dot{z}&=2A\, x^2.
\end{aligned}
\end{equation}
It is easy to check that the function $E(x,z)=(z-\rho)^2-2A\, x^2$ is an integral of system \eqref{eq:system_rho_mu_zero_restriction}. For $A<0$ the curves $E(x,z)=const$ are ellipses. When $\rho=\frac{1}{2}$ the ellipse $E(x,y)=\frac{1}{4}$ passes throw the points $O^+$ and $O$. Thus, the unstable separatrices of $O^+$ tend to $O$ as $t\to+\infty$. Namely, it is easy to see that
$$
x_0(t)=\frac{e^{t/2}}{\sqrt{2|A|}\, (1+e^t)}, \ \ y_0(t)=0, \ \ z_0(t)=\frac{1}{1+e^{t}}, \ \ t\in(-\infty,+\infty),
$$
is a solution of system \eqref{eq:system_rho_mu_zero_restriction}, asymptotic to $O^+$ as $t\to-\infty$ and asymptotic to $O$ as $t\to+\infty$. By the symmetry, $(x=-x_0(t), y=0, z=z_0(t))$ gives us the second unstable separatrix of $O^+$, which also tends to $O$ as $t\to+\infty$. Also, by the symmetry, the unstable separatrices of $O^-$ lie in the invariant plane $x=0$ and tend to $O$ as $t\to+\infty$ too.

Note that this heteroclinic connection splits with a non-zero velocity when $\rho$ changes. Indeed, the ellipse $E(x,z)=(1-\rho)^2$ passes throw the point $O^+$ and contains trajectories from the one-dimensional unstable manifold $W^u(O^+)$. The ellipse $E(x,z)=\rho^2$ passes through the point $O$ and contains trajectories from the two-dimensional stable manifold $W^s(O)$. These two ellipses intersect the cross-section $(z=\rho, x>0)$ at the points $(x,y,z)=((1-\rho)/\sqrt{2|A|},0,\rho)$ and $(x,y,z)=(\rho/\sqrt{2|A|},0,\rho)$. So, the distance between the intersection points is $(1-2\rho)/\sqrt{2|A|}$ and has a non-zero derivative with respect to $\rho$. Since the stable manifold $W^s(O)$ intersects the invariant plane $y=0$ transversely, the distance between these points plays the role of the distance between $W^s(O)$ and $W^u(O^+)$ in the cross section $z=\rho$. Therefore, we can conclude that the heteroclinic connection at $\rho=1/2, \omega=0, \mu=0$ splits with a non-zero velocity as $\rho$ varies.  Since the invariant manifolds $W^s(O)$ and $W^u(O^+)$ depend smoothly on parameters, this implies that the distance between $W^s(O)$ and $W^u(O^+)$ on the cross-section changes with a non-zero velocity, as $\rho$ varies, for all $(\rho,\omega,\sqrt{\mu})$ sufficiently close to $(1/2,0,0)$. This proves that the heteroclinic connection established below does indeed split with a non-zero velocity, as claimed in the theorem.

Further we denote $\hat{\rho}=1/2-\rho$, and study system \eqref{FlowNormalFormRescaled} when $\omega$, $\hat{\rho}$ and  $\mu>0$ are small. Then, in complex coordinates, system \eqref{FlowNormalFormRescaled} can be written in the form
$$
\begin{aligned}
\dot{u}&=-1/2\, u+  u^* z+(\hat{\rho}+i\omega)\, u+ \sqrt{\mu}\, f_1(u,z), \\
\dot{z}&=(A+iB)\, u^2+(A-iB)\, (u^*)^2  +\sqrt{\mu}\, f_2(u,z),
\end{aligned}
$$
where $f_1, f_2$ are given by \eqref{eq:high_order_terms}.

For this system, for small $\hat{\rho}$, $\omega$ and $\sqrt{\mu}$,  we look for a heteroclinic solution of the form
$$
u(t)=x_0(t)+U(t), \ \ \
z(t)=z_0(t)+Z(t),
$$
where $U(t)=X(t)+i  Y(t)$, and $X(t), Y(t), Z(t)$ are real-valued bounded smooth functions such that
\begin{equation}\label{boundary_condition}
\lim_{t\to\pm\infty} (X(t), Y(t), Z(t))=0.
\end{equation}
The functions $U(t), Z(t)$ solve the following non-autonomous system of differential equations:
$$
\begin{aligned}
\dot{U}&=-1/2\, U+z_0(t)\, U^*+x_0(t)\, Z+g_1(U,Z)+i g_2(U,Z),\\
\dot{Z}&=2(A+iB)\, x_0(t)\, U+2(A-iB)\, x_0(t)\, U^*+g_3(U,Z),
\end{aligned}
$$
where $g_i(U,Z)$ are real-valued functions given by the formulas
\begin{equation}\label{eq:g_def}
\begin{aligned}
&g_1(U,Z)+i g_2(U,Z)=U^* Z+(\hat{\rho}+i\omega)\, u+\sqrt{\mu}\, f_1(u,z), \\
&g_3(U,Z)=(A+iB)\, U^2+(A-iB)\, (U^*)^2+\sqrt{\mu}\, f_2(u,z).
\end{aligned}
\end{equation}
In real coordinates this system has the form
\begin{equation}\label{nonhomogeneous_system}
\begin{aligned}
\dot{X}&=(z_0(t)-1/2)\, X+x_0(t)\, Z+g_1(X,Y,Z),\\
\dot{Y}&=-(z_0(t)+1/2)\, Y+g_2(X,Y,Z),\\
\dot{Z}&=4(A\, X-B\, Y)\, x_0(t)+g_3(X,Y,Z).
\end{aligned}
\end{equation}
We need to find a solution of this system that satisfies boundary condition \eqref{boundary_condition}. To do this, we first solve the linear homogeneous system of equations
$$
\begin{aligned}
\dot{X}&=(z_0(t)-1/2)\, X+x_0(t)\, Z,\\
\dot{Y}&=-(z_0(t)+1/2)\, Y,\\
\dot{Z}&=4(A\, X-B\, Y)\, x_0(t).
\end{aligned}
$$
This system has three linearly independent solutions:
\begin{equation}\label{eq:fundamental_sol}
\begin{aligned}
(X_1,Y_1,Z_1)&=(\dot{x}_0(t) \, ,\,0 \, , \, \dot{z}_0(t)),\\
(X_2,Y_2,Z_2)&=\left(\dot{x}_0(t)\, \int_0^t\frac{ds}{x_0^2(s)}+\frac{1}{2x_0(t)} \, ,\,0 \, , \, \dot{z}_0(t)\, \int_0^t\frac{ds}{x_0^2(s)}\right),\\
(X_3,Y_3,Z_3)&=\left(c_1\, X_1+c_2\, X_2\ ,\ \frac{e^{-t}}{x_0(t)} \ ,\ c_1\, Z_1+c_2\, Z_2 \right),
\end{aligned}
\end{equation}
where $c_1$ and $c_2$ are given by the formulas
$$
c_1(t)=-\frac{4B}{A}\,\int\limits_t^{+\infty}  e^{-s}\, \frac{X_2(s)}{x_0(s)}\, ds, \ \
c_2(t)=\frac{4B}{A}\,\int\limits_t^{+\infty}  e^{-s} \, \frac{X_1(s)}{x_0(s)}\, ds.
$$
Note the following asymptotics as $t\to\pm\infty$:
\begin{equation}\label{eq:asymptotics}
\begin{aligned}
x_0(t)\sim e^{-|t|/2}, \ X_1(t)=O(e^{-|t|/2}), \ \ X_2(t)=O(t\, e^{-|t|/2}), \ \  X_3(t)= O(e^{-t}\, e^{-|t|/2}), \\
z_0(t)\sim e^{-|t|}, \ Z_1(t)=O(e^{-|t|}), \ \ Z_2(t)\to\mp 2|A|, \ \ Z_3(t)=O(e^{-t}), \ \
Y_3(t)\sim e^{-t}\, e^{|t|/2}.
\end{aligned}
\end{equation}

Now we write an integral equation for the sought heteroclinic solution of system \eqref{nonhomogeneous_system}. Denote
$$
\Phi(t)=
\begin{pmatrix}
X_1& X_2 & X_3 \\
0& 0 & Y_3 \\
Z_1& Z_2 & Z_3 \\
\end{pmatrix}, \ \
F(t)=
\begin{pmatrix}
g_1(X(t),Y(t),Z(t)) \\
g_2(X(t),Y(t),Z(t)) \\
g_3(X(t),Y(t),Z(t))
\end{pmatrix}.
$$
By the variation of constants method, we obtain
$$
(X,Y,Z)^{T}=\Phi(t)\, \int \Phi^{-1}(s)\, F(s)\, ds.
$$
Since
$$
\Phi^{-1}(t)=\frac{1}{Ax_0(t)}\,
\begin{pmatrix}
-Z_2 & (-X_2 Z_3+Z_2 X_3)/Y_3 & X_2 \\
Z_1 & (X_1 Z_3-Z_1 X_3)/Y_3 & -X_1 \\
0 & (-X_1 Z_2+Z_1 X_2)/Y_3 & 0 \\
\end{pmatrix},
$$
we have
\begin{equation}\label{eq:G_def}
\Phi^{-1}(t)\, F(s)=
\frac{1}{A x_0(t)}\,
\begin{pmatrix}
-Z_2\, g_1+X_2\, g_3+(-X_2 Z_3/Y_3+Z_2 X_3/Y_3)\, g_2\\
Z_1\, g_1-X_1\, g_3 +(X_1 Z_3/Y_3-Z_1 X_3/Y_3)\, g_2\\
(-X_1 Z_2+Z_1 X_2)\, g_2/Y_3 \\
\end{pmatrix}:=\begin{pmatrix}
G_1 \\
G_2 \\
G_3 \\
\end{pmatrix}.
\end{equation}

Therefore,
\begin{equation}\label{eq:general_solution}
\begin{aligned}
X(t)&=X_1(t)\, \int\limits_{t_1}^{t} G_1(s)\, ds+X_2(t)\, \int\limits_{t_2}^{t} G_2(s)\, ds+X_3(t)\, \int\limits_{t_3}^{t} G_3(s)\, ds, \\
Y(t)&= Y_3(t)\, \int\limits_{t_3}^t G_3(s)\, ds, \\
Z(t)&=Z_1(t)\, \int\limits_{t_1}^t G_1(s)\, ds+Z_2(t)\, \int\limits_{t_2}^t G_2(s)\, ds+Z_3(t)\, \int\limits_{t_3}^t G_3(s)\, ds,
\end{aligned}
\end{equation}
where different integration limits $t_1, t_2, t_3$ correspond to different solutions of system \eqref{nonhomogeneous_system}. The unstable separatrix of $O^+$ corresponds to a solution such that $X, Y, Z$ tend to $0$ as $t\to-\infty$, and thus satisfies the integral equations
\begin{equation}\label{eq:unstable_manifold}
\begin{aligned}
X(t)&=X_1(t)\, \int\limits_{0}^t G_1(s)\, ds+X_2(t)\, \int\limits_{-\infty}^t G_2(s)\, ds+X_3(t)\, \int\limits_{-\infty}^t G_3(s)\, ds, \\
Y(t)&= Y_3(t)\, \int\limits_{-\infty}^t G_3(s)\, ds, \\
 Z(t)&=Z_1(t)\,  \int\limits_{0}^t G_1(s)\, ds+Z_2(t)\, \int\limits_{-\infty}^t G_2(s)\, ds+Z_3(t)\, \int\limits_{-\infty}^t G_3(s)\, ds.
\end{aligned}
\end{equation}
We took $t_2=-\infty$ because $Z_2(t)$ tends to a non-zero constant as $t\to-\infty$, and $t_3=-\infty$ because $Y_3(t)\to \infty$ as $t\to-\infty$. We assume $t_1=0$ because solutions corresponding to different $t_1$ must give us the same trajectory of system \eqref{FlowNormalFormRescaled} (the unstable separatrix) up to a time shift.

The solutions corresponding to trajectories of $W^s(O)$ satisfy the integral equations
\begin{equation}\label{eq:stable_manifold}
\begin{aligned}
X(t)&=X_1(t)\, \left(C_1+ \int\limits_{0}^t G_1(s)\, ds\right)-X_2(t)\, \int\limits_{t}^{+\infty} G_2(s)\, ds+X_3(t)\, \left(C_3+ \int\limits_{-\infty}^t G_3(s)\, ds\right), \\
Y(t)&= Y_3(t)\, \left(C_3+ \int\limits_{-\infty}^t G_3(s)\, ds\right), \\
 Z(t)&=Z_1(t)\,\left(C_1+ \int\limits_{0}^t G_1(s)\, ds\right)-Z_2(t)\, \int\limits_{t}^{+\infty} G_2(s)\, ds+Z_3(t)\,\left(C_3+ \int\limits_{-\infty}^t G_3(s)\, ds\right),
\end{aligned}
\end{equation}
with arbitrary constants $C_1, C_3$.

The sought heteroclinic solution connecting $O^+$ and $O$ must satisfy both \eqref{eq:unstable_manifold} and \eqref{eq:stable_manifold}. Therefore, it satisfies the following system of integral equations
\begin{equation}\label{intergal_system}
\begin{aligned}
X(t)&=X_1(t)\, \int\limits_{0}^t G_1(s)\, ds+X_2(t)\, \int\limits_{-\infty}^t G_2(s)\, ds+X_3(t)\, \int\limits_{-\infty}^t G_3(s)\, ds, \\
Y(t)&= Y_3(t)\, \int\limits_{-\infty}^t G_3(s)\, ds, \\
 Z(t)&=
 \begin{cases}Z_1(t)\, \int\limits_{0}^t G_1(s)\, ds+Z_2(t)\, \int\limits_{-\infty}^t G_2(s)\, ds+Z_3(t)\, \int\limits_{-\infty}^t G_3(s)\, ds,  \ \  \text{for $t\leq 0$}, \\
Z_1(t)\, \int\limits_{0}^t G_1(s)\, ds-Z_2(t)\, \int\limits_t^{+\infty} G_2(s)\, ds+Z_3(t)\, \int\limits_{-\infty}^t G_3(s)\, ds \ \  \text{for $t> 0$},
\end{cases}
\end{aligned}
\end{equation}
and
\begin{equation}\label{intergal_condition}
\int\limits_{-\infty}^{+\infty} G_2(s)\, ds=0.
\end{equation}

Let us show that for small $\hat{\rho}$, $\omega$ and $\sqrt{\mu}$ system \eqref{intergal_system} has a unique solution in the space of vector functions $(X(t),Y(t),Z(t))$ such that $X(t), Y(t)$ are continuous for all $t$ and $Z(t)$ is also continuous except for the point $t=0$, and $(X(t),Y(t),Z(t))$ are bounded in the weighted norm
$$
\|(X,Y,Z)\|_{w}=\max \left(\sup_{t\in \mathbb{R}}\, \frac{|X(t)|}{e^{-|t|/4}}, \ \ \sup_{t\in \mathbb{R}}\, \frac{|Y(t)|}{e^{-|t|/4}}, \ \ \sup_{t\in \mathbb{R}}\, \frac{|Z(t)|}{e^{- |t|/2}}\right).
$$
By the contraction mapping principle, it is enough to show that there exists a small constant $\varepsilon$ such that for any sufficiently small $(\hat{\rho},\omega,\sqrt{\mu})$, the right-hand sides of equations \eqref{intergal_system} define a contracting map of the ball $\mathcal{B}_{\varepsilon}=\{\|X,Y,Z\|_{w}\leq \varepsilon\}$ into itself. Indeed, let us first check that the set $\mathcal B_{\varepsilon}$ is invariant under these transformations for some $\varepsilon$. Consider a function $(X,Y,Z)\in \mathcal B_{\varepsilon}$ and let $(\bar{X},\bar{Y},\bar{Z})$ be its image. Since in \eqref{eq:high_order_terms} the function $f_1$ vanishes at $u=u^*=0$ and $f_2$ vanishes at the equilibrium states $(u,u^*,z)=(0,0,\pm 1)$ and $(u,u^*,z)=(0,0,0)$ along with its derivatives with respect to $u$ and $u^*$, it follows from \eqref{eq:asymptotics}, that for $(X,Y,Z)\in \mathcal{B}_{\varepsilon}$
\begin{equation}\label{eq:f_estimates}
\begin{aligned}
f_1(u,u^*,z)=O(|u|)=O(e^{-|t|/4}),&  \ \ \
f_2(u,u^*,z)=O(|z|\, |z-1|+|u|^2)= O(e^{-|t|/2}), \\
\frac{\partial f_2}{\partial (u,u^*)}&=O(|z|\, |z-1|+|u|)=O(e^{-|t|/4}).
\end{aligned}
\end{equation}
Therefore, by \eqref{eq:g_def}, \eqref{eq:asymptotics}, and \eqref{eq:f_estimates},
$$
g_{1,2}(X,Y,Z)= O(\varepsilon^2,\hat{\rho},\omega,\sqrt{\mu})\, e^{-|t|/4}, \ \ \  \
g_3(X,Y,Z)=O(\varepsilon^2,\sqrt{\mu})\, e^{-|t|/2}.
$$
Thus, by \eqref{eq:asymptotics} and \eqref{eq:G_def}, the functions $G_{1,2,3}$ satisfy the estimates
$$
G_1(t)=O(\varepsilon^2,\hat{\rho},\omega,\sqrt{\mu})\, e^{|t|/4}, \ \ \
G_2(t)=O(\varepsilon^2,\hat{\rho},\omega,\sqrt{\mu})\, e^{- |t|/2}, \ \ \
G_3(t)=O(\varepsilon^2,\hat{\rho},\omega,\sqrt{\mu})\, e^{t}\, e^{-3|t|/4}.
$$
In particular, these estimates imply that the integrals in \eqref{intergal_system} are well defined and the following asymptotics hold
$$
\int\limits_{0}^t G_1(s)\, ds=O(\varepsilon^2,\hat{\rho},\omega,\sqrt{\mu})\, e^{|t|/4}, \ \ \ \ \ \int\limits_{-\infty}^t G_2(s)\, ds \stackrel{t\to -\infty}{=}O(\varepsilon^2,\hat{\rho},\omega,\sqrt{\mu}) \, e^{-|t|/2},
$$
$$
\int\limits_t^{+\infty} G_2(s)\, ds \stackrel{t\to +\infty}{=} O(\varepsilon^2,\hat{\rho},\omega,\sqrt{\mu})\, e^{-|t|/2}, \ \ \ \ \ \
\int\limits_{-\infty}^t G_3(s)\, ds =O(\varepsilon^2,\hat{\rho},\omega,\sqrt{\mu}) \, e^{t}\, e^{-3|t|/4}.
$$
The obtained estimates imply, together with \eqref{eq:asymptotics}, that the image functions $(\bar{X},\bar{Y},\bar{Z})$ satisfy
$$
\bar{X}(t)=O(\varepsilon^2,\hat{\rho},\omega,\sqrt{\mu})\, e^{-|t|/4},\ \ \ \bar{Y}(t)=O(\varepsilon^2,\hat{\rho},\omega,\sqrt{\mu})\, e^{-|t|/4}, \ \ \ \bar{Z}(t)=O(\varepsilon^2,\delta_1,\delta_2)\, e^{-|t|/2}.
$$
For any sufficiently small $\varepsilon$, if $\hat{\rho}$, $\omega$, and $\sqrt{\mu}$ are small enough, then all the coefficients $O(\varepsilon^2,\hat{\rho},\omega,\sqrt{\mu})$ in the above formulas are less than $\varepsilon$ in the absolute value. It is also obvious that the functions $\bar{X},\bar{Y}$ are continuous and $\bar{Z}$ can only have a discontinuity at $t=0$. Therefore, the vector function $(\bar{X},\bar{Y},\bar{Z})$ belongs to the ball $\mathcal{B}_{\varepsilon}$, as required.

Now we check that the transformation $(X,Y,Z)\to (\bar{X},\bar{Y},\bar{Z})$ is contracting in $\mathcal{B}_{\varepsilon}$. For a pair of vector functions $(X_1, Y_1,Z_1), (X_2, Y_2,Z_2)\in \mathcal{B}_{\varepsilon}$, we have, by \eqref{eq:g_def}, \eqref{eq:asymptotics} and \eqref{eq:f_estimates},
$$
|g_1(X_1, Y_1,Z_1)-g_1(X_2, Y_2,Z_2)|<(2\varepsilon+|\hat{\rho}|+|\omega|+O(\sqrt{\mu}))\, \|(X_1, Y_1,Z_1)-(X_2, Y_2,Z_2)\|_{w}\, e^{-|t|/4},
$$
$$
|g_2(X_1, Y_1,Z_1)-g_2(X_2, Y_2,Z_2)|<(2\varepsilon+|\hat{\rho}|+|\omega|+O(\sqrt{\mu}))\, \|(X_1, Y_1,Z_1)-(X_2, Y_2,Z_2)\|_{w}\, e^{-|t|/4},
$$
$$
|g_3(X_1, Y_1,Z_1)-g_3(X_2, Y_2,Z_2)|<(8(|A|+|B|)\, \varepsilon+O(\sqrt{\mu}))\, \|(X_1, Y_1,Z_1)-(X_2, Y_2,Z_2)\|_{w}\, e^{-|t|/2}.
$$
By \eqref{eq:asymptotics}, \eqref{eq:G_def}, and \eqref{intergal_system}, we obtain
$$
|\bar{X}_1(t)-\bar{X}_2(t)|<O(\varepsilon,\hat{\rho},\omega,\sqrt{\mu})\, \|(X_1, Y_1,Z_1)-(X_2, Y_2,Z_2)\|_{w}\, e^{-|t|/4} ,
$$
$$
|\bar{Y}_1(t)-\bar{Y}_2(t)|<O(\varepsilon,\hat{\rho},\omega,\sqrt{\mu})\, \|(X_1, Y_1,Z_1)-(X_2, Y_2,Z_2)\|_{w}\, e^{-|t|/4},
$$
$$
|\bar{Z}_1(t)-\bar{Z}_2(t)|<O(\varepsilon,\hat{\rho},\omega,\sqrt{\mu})\, \|(X_1, Y_1,Z_1)-(X_2, Y_2,Z_2)\|_{w}\, e^{-|t|/2}.
$$
These inequalities mean
$$
\|(\bar{X}_1,\bar{Y}_1,\bar{Z}_1)-(\bar{X}_2,\bar{Y}_2,\bar{Z}_2)\|_{w}=O(\varepsilon,\hat{\rho},\omega,\sqrt{\mu}) \, \|(X_1, Y_1,Z_1)-(X_2, Y_2,Z_2)\|_{w},
$$
which means the contraction. Thus, system \eqref{intergal_system} has a unique small solution in $\mathcal{B}_{\varepsilon}$, as required.

Note that we can choose $\varepsilon=O(\hat{\rho},\omega,\sqrt{\mu})$ and, therefore, the solution $(X,Y,Z)$ is $O(\hat{\rho},\omega,\sqrt{\mu})$. Recall that condition \eqref{intergal_condition} defines the parameter values for which this solution corresponds to the heteroclinic solution
$$
(u(t),z(t))=(x_0(t)+X(t)+i Y(t), z_0(t)+Z(t))=(x_0(t),z_0(t))+O(\hat{\rho},\omega,\sqrt{\mu})
$$
of system \eqref{FlowNormalFormRescaled}. Due to \eqref{eq:high_order_terms}, \eqref{eq:fundamental_sol}, and \eqref{eq:G_def}, the integral in  \eqref{intergal_condition} can be written as
$$
\int\limits_{-\infty}^{+\infty} G_2(s)\, ds=\int\limits_{-\infty}^{+\infty} \frac{Z_1}{Ax_0}\, (\hat{\rho} \, x_0+\sqrt{\mu}\, {\rm Re}\, f_1(x_0,z_0))\, ds - \int\limits_{-\infty}^{+\infty} \frac{X_1}{Ax_0}\, \sqrt{\mu} \, f_2(x_0,z_0)\, ds+
$$
$$
+\int\limits_{-\infty}^{+\infty}
\frac{X_1 Z_3-Z_1 X_3}{Ax_0 Y_3}\, (\omega \, x_0+\sqrt{\mu}\, {\rm Im}\, f_1(x_0,z_0)))+O(\mu+\hat{\rho}^2+\omega^2)=
$$
$$
=\frac{1}{A}\, \hat{\rho}+ \frac{k_1}{A}\, \sqrt{\mu}+\frac{k_2}{A}\, \omega+O(\mu+\hat{\rho}^2+\omega^2),
$$
where
$$
k_1=\frac{\sqrt{b_1}}{12|a_0|}+ \left({\rm Re}(a_1)+\frac{{\rm Im}(a_1) B}{3A} \right)\,\frac{1}{12|A|\,|a_0|}+\left({\rm Re}(a_2)-\frac{{\rm Im}(a_2) B}{6 A}\right)\, \frac{1}{3|a_0|\,\sqrt{|b_1|}}+
$$
$$
+\left( {\rm Re}(a_3/a_0^2)+\frac{{\rm Im}(a_3/a_0^2)B}{3A}\right)\, \frac{|a_0|}{12|A|}+\frac{b_2}{24|A||a_0|},
$$
$$
k_2=\frac{(\pi^2-9)B}{3A}.
$$
So, condition \eqref{intergal_condition} gives us the sought  bifurcation surface given by equation \eqref{eq:het_curve}.
\end{proof}

\section{Shilnikov-like criterion for the birth of pseudohyperbolic attractors from the heteroclinic connection}\label{Shilnikov_criterion}

In this section we study bifurcations of the four-winged heteroclinic connection when the equilibrium state $O$ is a dicritical node on its stable manifolds, see Fig. \ref{HeteroclinicConnection}b. We prove that, under additional assumptions, perturbations of a system such that the heteroclinic connection splits and $O$ becomes a saddle-focus lead to the appearance of the Lorenz and Sim\'o attractors. Then, we check that the heteroclinic connection in system \eqref{FlowNormalFormGeneral} provided by Theorems \ref{th2} and \ref{th5} satisfies the required assumptions. As a result, we prove Theorem \ref{th3}.

Let $\dot{w}=G_{\theta,\delta}(w)$, where $w=(x,y,z)\in\mathbb{R}^3$, be a two-parameter family of smooth vector fields depending smoothly on the parameters $\theta$ and $\delta$. Assume that the vector fields are invariant with respect to the $\mathbb{Z}_4$-symmetry generated by the matrix
$$
S=
\begin{pmatrix}
0 & -1 & 0\\
1 & 0 & 0\\
0 & 0 & -1
\end{pmatrix}.
$$
Due to the symmetry $S$,  the $z$-axis is an invariant line of the vector field and the point $O(0,0,0)$ is a symmetric equilibrium state. We suppose that
\begin{itemize}
\item[(B1)] the equilibrium state $O(0,0,0)$ has a two-dimensional stable manifold and a one-dimensional unstable manifold which coincides with a segment of the $z$-axis. For $\theta\neq 0$ the equilibrium state $O$ is a saddle-focus, and for $\theta=0$ it is a dicritical node on the stable manifold.
\end{itemize}
In a neighborhood of the equilibrium state $O$, one can introduce coordinates in such a way that the system, after an appropriate rescaling of time, can be written locally in the form (see \cite{book1, book2})
\begin{equation}\label{LinearFlowO}
\begin{aligned}
\dot{x}&=- x-\theta \, y+f_{11}(x,y,z)\, x+f_{12}(x,y,z)\, y,  \\
\dot{y}&=\theta \, x-  y+f_{21}(x,y,z)\, x+f_{22}(x,y,z)\, y,\\
\dot{z}&=\hat{\nu} \, z,
\end{aligned}
\end{equation}
where $\hat{\nu}>0$, and  the non-linear terms satisfy the identities
\begin{equation}\label{LinearFlowO_high_terms}
f_{ij}(x,y,0)\equiv 0, \ \ f_{ij}(0,0,z)\equiv 0.
\end{equation}
The functions $f_{ij}$, as well as $\hat{\nu}$, are smooth functions of small parameters $\theta$, and $\delta$. Note that one can make the coordinate transformation, which brings the system to the form \eqref{LinearFlowO}, such that the system remains $S$-symmetric. 

We further assume that there is a pair of symmetric equilibrium states $O^-$ and $O^+$ on the $z$-axis and the following holds
\begin{itemize}
\item[(B2)] the equilibrium states $O^-$ and $O^+$ are saddles with a one-dimensional unstable manifold and a two-dimensional stable manifold whose the leading direction coincides with the $z$-axis.
\end{itemize}
We also assume that there are no other equilibrium states lying on the $z$-axis between $O^-$ and $O^+$. Therefore, the unstable manifold of the point $O$ contains the points $O^-$ and $O^+$ in its closure.

In a neighborhood of the saddle $O^-$ one can introduce local coordinates $(x_1,y_1,z_1)$ and change time (see \cite{book1,book2}) such that the system takes the form
\begin{equation}\label{LinearFlowO-}
\begin{aligned}
\dot{x}_1&=-\nu_{ss}\, x_1+g_{11}(y_1,z_1)\, z_1+ g_{12}(x_1,y_1,z_1)\, x_1,  \\
\dot{y}_1&=y_1,\\
\dot{z}_1&=-\nu\, z_1+g_{21}(y_1,z_1)\, z_1+ g_{22}(x_1,y_1,z_1)\, x_1,
\end{aligned}
\end{equation}
where $0<\nu<\nu_{ss}$ and
\begin{equation}\label{LinearFlowO-_high_terms}
g_{21}(0,z_1)\equiv 0, \ \   g_{22}(x_1,0,z_1)\equiv 0, \ \  g_{11}(y_1,0)\equiv 0, \ \  g_{21}(y_1,0)\equiv 0.
\end{equation}

The functions $g_{ij}$ and the coefficients $\nu$, $\nu_{ss}$ are smooth functions of the parameters $\theta$, $\delta$, and the coordinate transformation is chosen in such a way that system \eqref{LinearFlowO-} is $S^2$-symmetric, i.e., it is symmetric with respect to $x_1\to -x_1$, $y_1\to -y_1$.

We assume that the coefficients $\nu$ and $\hat{\nu}$ satisfy the following inequality
\begin{itemize}
\item[(B3)] $0<\hat{\nu}<\nu<(1+\hat{\nu})/2<1$.
\end{itemize}
For convenience, the local coordinates are chosen such that the positive $z_1$ corresponds to positive $z$. By the symmetry, near the equilibrium state $O^+$ the system also has form \eqref{LinearFlowO-} in appropriate local coordinates.

The one-dimensional unstable manifold $W^u(O^-)$ (as well as $W^u(O^+)$) consists of two trajectories (separatrices) which we denote by $\Gamma_1^-$ and $\Gamma_2^-$ (resp., $\Gamma_1^+$ and $\Gamma_2^+$). We assume that the following holds
\begin{itemize}
\item[(B4)] for $\delta=0$ the unstable separatrices $\Gamma_1^-$, $\Gamma_2^-$ (and, by the symmetry, $\Gamma_1^+$, $\Gamma_2^+$) lie in the stable manifold $W^s(O)$. When $\delta$ changes, the corresponding heteroclinic connections split with a non-zero velocity.
\end{itemize}
So, for $\delta=\theta=0$, the system $\dot{w}=G_{\delta,\theta}(w)$ has the four-winged heteroclinic connection as shown in Fig. \ref{HeteroclinicConnection}b. The parameter $\delta$ plays the role of the distance between the unstable manifold $W^u(O^-)$ and the stable manifold $W^s(O)$. More precisely, in a small neighborhood of the point $O$ we consider a cylindrical cross-section $\Pi_{cyl}$, corresponding to a constant value of $x^2+y^2$. We perform a linear scaling of the coordinates near $O$ so that the cross-section is given by
$$\Pi_{cyl}=\{(r,\phi,z) \; |\; r=1,\; |z|<1 \},
$$
where $x=r \cos\phi$,  $y=r \sin \phi$. We fix the choice of the coordinates by letting the coordinates of the intersection point of $\Gamma_1^-$ with the cross-section $\Pi_{cyl}$ be equal to $(z,\phi)=(\delta,0)$. This is achieved by a scaling $z$ and a linear rotation in the $(x,y)$-plane. Obviously, this does not change the form of system \eqref{LinearFlowO}, \eqref{LinearFlowO_high_terms}. By the symmetry, the separatrix $\Gamma_2^-$ intersects the cross-section $\Pi_{cyl}$ at the point $(z,\phi)=(\delta,\pi)$.

For the equilibrium $O^-$, there is an invariant extended unstable smooth manifold $W^{uE}(O^-)$ tangent to the $(y_1,z_1)$-plane at the point $O^-$. This manifold contains the unstable separatrices $\Gamma_{1,2}^-$ of $O^-$. We assume that
\begin{itemize}
\item[(B5)] For $\delta=\theta=0$ the extended unstable manifold $W^{uE}(O^-)$ and the stable manifold $W^s(O)$ transversely intersect at points of $\Gamma_{1,2}^-$.
\end{itemize}

The two-dimensional stable manifolds $W^s(O^-)$ and $W^s(O^+)$ contain the $z$-axis, and these manifolds are foliated by strong stable foliations \cite{book1, book2}. Denote by $V(z)$ the tangent subspace to the strong stable foliations at a point $(0,0,z)$, where $0<|z|<1$.
If the equilibrium $O$ is a dicritical node on its stable manifolds, then there are two limit subspaces
$$V^-=\lim_{z\to -0} V(z), \ \ \  V^+=\lim_{z\to +0} V(z).$$
Due to the symmetry, the limit subspaces are orthogonal to each other. Denote by $\Omega^{\pm}\in (-\pi/2,\pi/2]$ the angle between the subspace $V^-$ (or $V^+$) and the $x$-axis for $\theta=\delta=0$. As we mentioned, the symmetry implies that $|\Omega^--\Omega^+|=\pi/2$.

\begin{figure}[h]
    \centering
    \includegraphics[width=0.5\linewidth]{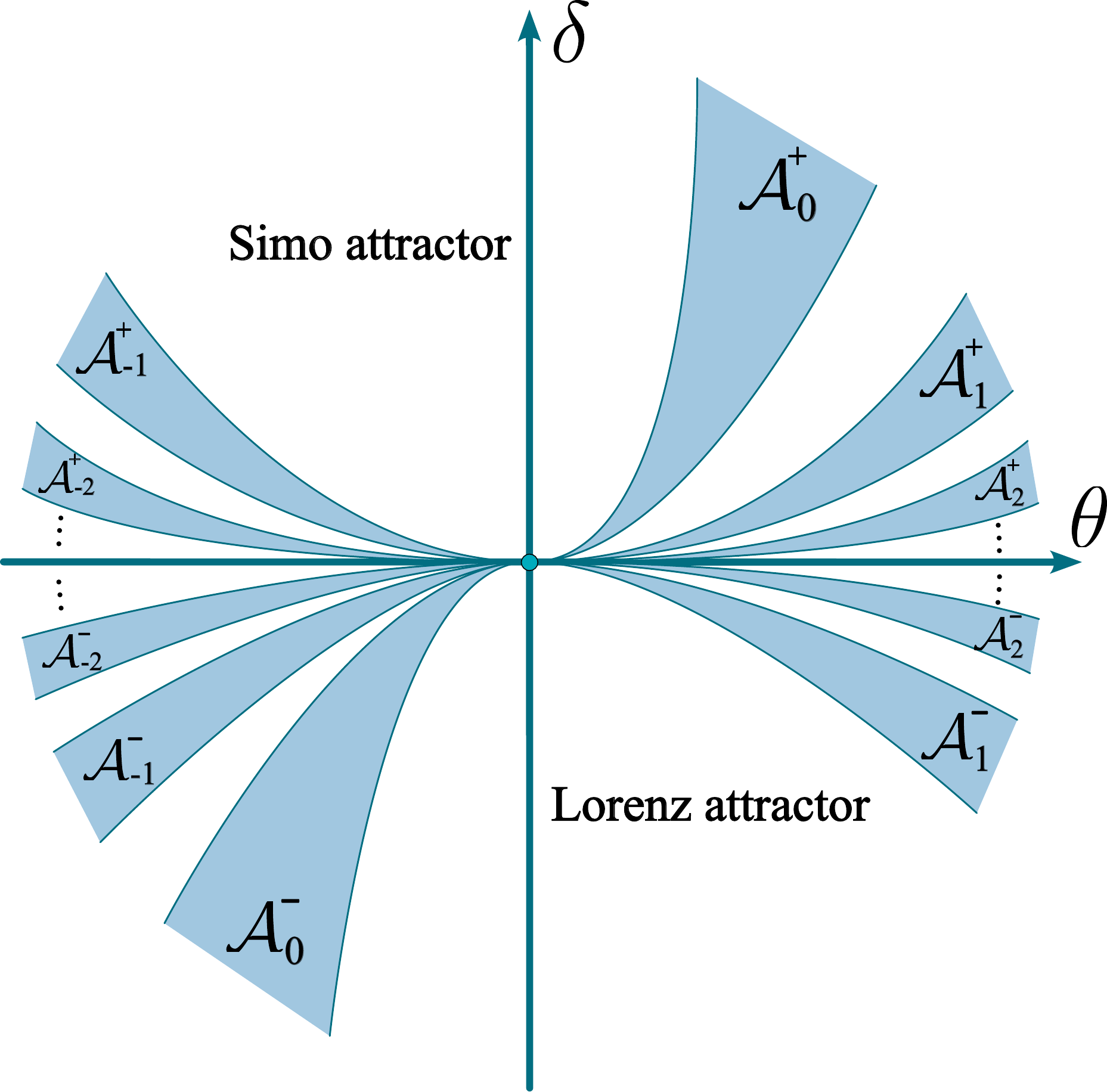}
    \caption{Regions of the existence of the attractors near the stem point. For $\delta>0$ the regions $\mathcal{A}^+_k$ correspond to the Sim\'o angel; for $\delta<0$ the regions $\mathcal{A}^-_k$ correspond to the Lorenz attractor.}
    \label{StemPoint}
\end{figure}

\begin{theorem}\label{th6}
Let a family of vector fields satisfy assumptions B1-B5. Then, in the half plane $(\theta,\delta<0)$, there exists a countable sequence of open disjoint regions $\mathcal {A}^-_k$ corresponding to the existence of symmetric pairs of Lorenz attractors, and, in the half plane $(\theta,\delta>0)$, there exists a countable sequence of open disjoint regions $\mathcal {A}^+_k$ corresponding to the existence of four-winged Sim\'o angels, see Fig. \ref{StemPoint}.

For the sequence $\mathcal {A}^{-}_k$, the index $k$ runs through all integer values except for $k=0$ if $\Omega^{-}=0$. The region $\mathcal {A}^{-}_k$ lies in the quadrant $\theta\, (\Omega^{-}-\pi k)<0$, adjoins the point $(\theta,\delta)=(0,0)$ and is given by
\begin{equation}\label{eq:domain_A-_equation}
K_1^{-}\, \exp\left(-\frac{\nu|\Omega^{-}-\pi k|}{|\theta|}\right)<\delta<K_2^{-}\, \exp\left(-\frac{\nu|\Omega^{-}-\pi k|}{|\theta|}\right)
\end{equation}
for some positive constants $K_{1,2}^-$ independent of $k$.

Similarly, for the sequence $\mathcal {A}^{+}_k$, the index $k$ runs through all integer values except for $k=0$ if $\Omega^{+}=0$. The region $\mathcal {A}^{+}_k$ lies in the quadrant $\theta\, (\Omega^{+}-\pi k)<0$, adjoins the point $(\theta,\delta)=(0,0)$ and is given by
\begin{equation}\label{eq:domain_A+_equation}
K_1^{+}\, \exp\left(-\frac{\nu|\Omega^{+}-\pi k|}{|\theta|}\right)<\delta<K_2^{+}\, \exp\left(-\frac{\nu|\Omega^{+}-\pi k|}{|\theta|}\right)
\end{equation}
for some positive constants $K_{1,2}^+$ independent of $k$.
\end{theorem}

The statement of the theorem follows from Proposition \ref{th:Poincare_map} and Theorem \ref{th7} (see also discussion in Section \ref{Pseudohyperbolic_attractors}) which we give below. Proposition \ref{th:Poincare_map} describes the Poincar\'e map (or in the case of the attractor Sim\'o, the half-map) on a small cross-section, which plays the role of the cross-section $\Pi^-$ in assumption A5 in Section \ref{Pseudohyperbolic_attractors}.

Namely, near the point $O^-$ we choose a small cross-section $\Pi^-$ such that in the local coordinates it has the form (after rescaling $z_1$)
$$\Pi^-=\{(x_1,y_1,z_1)\;|\; z_1=1\}.$$
Let $\Pi^+=S(\Pi^-)$ be the symmetric cross-section in a neighborhood of the point $O^+$.  We construct the Poincar\'e map $T$ along trajectories of the system that maps the section $\Pi^-$ to $\Pi^+$ in the case $\delta>0$ and to $\Pi^-$ in the case $\delta<0$.

First, in a neighborhood of the point $O^-$, we choose two cross-sections $\Pi_1$ and $\Pi_2$ (see Fig.~\ref{Poincare_map}), which are given in the local coordinates by
$$\Pi_1=\{(x_1,y_1,z_1)\;|\; y_1=1\}, \ \ \Pi_2=\{(x_1,y_1,z_1)\;|\; y_1=-1\},$$
and consider the local Poincar\'e map $T_1:\Pi^-\backslash \{y_1=0\}\to \Pi_1\cup \Pi_2$. Namely, for sufficiently small $y_1\neq 0$, a trajectory starting at a point $(x_1,y_1,1)$ on the cross-section $\Pi^-$ intersects the cross-section $\Pi_1$ or $\Pi_2$, depending on the sign of $y_1$, at the point $(\bar{x}_1,{\rm sign}(y_1),\bar{z}_1)$, where $\bar{x}_1$, $\bar{z}_1$ are given by the formulas (\cite{book1,book2})
\begin{equation}\label{eq:T1_map}
T_1:\;
\begin{aligned}
\bar{x}_1&=o(|y_1|^{\tilde{\nu}}),\\
\bar{z}_1&=|y_1|^{\nu}+o(|y_1|^{\tilde{\nu}}),
\end{aligned}
\end{equation}
where $\nu<\tilde{\nu}<\min(2\nu,\nu_{ss})$. The expressions $o(|y_1^{\tilde{\nu}}|)$ refer to functions of $(x_1,y_1)$ whose derivatives containing $r$ differentiations with respect to $y_1$ are estimated as $o(|y_1|^{\tilde{\nu}-r})$.

For a small $\delta$, in a neighborhood of $\Gamma^-_1$ and $\Gamma^-_2$ we can consider the global Poincar\'e map $T_2:\Pi_1\cup \Pi_2 \to \Pi_{cyl}$, which is a local diffeomorphism and can be  written in the form
\begin{equation}\label{eq:T2_map}
\begin{array}{c}
T_2 |_{\Pi_1}: \;
\begin{aligned}
\bar{\phi}&=a_{11}\, x_1+a_{12}\, z_1 +O_2,\\
\bar{z}&=\delta+a_{21}\, x_1+a_{22}\, z_1+O_2,
\end{aligned} \\ \\
T_2 |_{\Pi_2}: \;
\begin{aligned}
\bar{\phi}&=\pi-a_{11}\, x_1+a_{12}\, z_1 +O_2,\\
\bar{z}&=\delta-a_{21}\, x_1+a_{22}\, z_1+O_2,
\end{aligned}
\end{array}
\end{equation}
where $a_{ij}$ are smooth functions of the parameters $\theta$ and $\delta$. Note that $T_2|_{\Pi_1}$ and $T_2|_{\Pi_2}$ are related by the symmetry $S^2$. Note that assumption B5 implies that  $a_{22}\neq 0$ for sufficiently small $\theta$ and $\delta$. Indeed, the manifold $W^s_{loc}(O)$ is $\bar{z}=0$ and, when condition \eqref{LinearFlowO-_high_terms} is satisfied, $W^{uE}_{loc}(O^-)$ is tangent to $x_1=0$, see \cite{book1}. The transversality of $W^s_{loc}(O)$ and $T_2(W^{uE}_{loc}(O^-))$  is obviously equivalent to $\partial \bar{z}/\partial z\neq 0$ for $(x_1,z_1)=(0,0)$.

Consider a neighborhood of the point $O$. Denote by $\Pi_{bot}$ and $\Pi_{top}$ the bottom and the top of the cylinder, respectively, i.e.,
$$
\Pi_{bot}=\{(r,\phi,z) \; |\; r<1,\; z=-1 \}, \ \ \Pi_{top}=\{(r,\phi,z) \; |\; r<1,\; z=1 \}.
$$
The local Poincar\'e map $T_3:\Pi_{cyl}\backslash \{z=0\}\to \Pi_{bot}\cup \Pi_{top}$ is given by (\cite{book1,book2})
\begin{equation}\label{eq:T3_map}
T_3:\;
\begin{aligned}
\bar{x}&=e^{-\tau}\, \cos(\phi+\theta \tau)+O(e^{-2\tau}), \\
\bar{y}&=e^{-\tau}\, \sin(\phi+\theta \tau)+O(e^{-2\tau}),
\end{aligned}
\end{equation}
where the transition time $\tau$ from the cross-section $\Pi_{cyl}$ to $\Pi_{bot}\cup \Pi_{top}$ is equal to
$$
\tau=\frac{1}{\hat{\nu}}\, \ln\frac{1}{|z|}.
$$
The map $T_3$ maps the point $(\rho=1,\phi,z)\in \Pi_{cyl}$ to the point $(\bar{x},\bar{y},\bar{z}={\rm sign}(z))\in \Pi_{bot}\cup \Pi_{top}$. Here, the expressions $O(e^{-2\tau})$ denote smooth functions of $(\tau,\phi)$ which are $O(e^{-2\tau})$ along with their derivatives.

Consider the global Poincar\'e map $T_4:\Pi_{bot}\cup \Pi_{top}\to \Pi_-\cup\Pi_+$. This map, restricted to the cross-section $\Pi_{bot}$, is a local diffeomorphism that maps the point $(x,y,z)=(0,0,-1)\in \Pi_{bot}$ to the point $(x_1,y_1,z_1)=(0,0,1)\in \Pi^-$.  Therefore, it can be written in the form

\begin{equation}\label{eq:T4_map}
T_4 |_{\Pi_{bot}}: \;
\begin{aligned}
\bar{x}_1&=b_{11}\, x+b_{12}\, y +O_2,\\
\bar{y}_1&=b_{21}\, x+b_{22}\, y+O_2,
\end{aligned}
\end{equation}
where $b_{ij}$ are smooth functions of the parameters $\theta$ and $\delta$. One can check that when conditions \eqref{LinearFlowO_high_terms} and \eqref{LinearFlowO-_high_terms} are fulfilled, $b_{21}/b_{22}=\tan(\Omega^-)$ for $\theta=\delta=0$ (we assume $\Omega^-=\pi/2$ in the case $b_{22}=0$). So, we
denote $B_i=\sqrt{b_{i1}^2+b_{i2}^2}$ and choose $\omega^-_i$, smoothly depending on the parameters,  such that $\tan(\omega^-_i)=b_{i1}/b_{i2}$
(we put $\omega^-_i=\pi/2$ if $b_{i2}=0$). We also fix the choice of $\omega_2^-$ by letting $\omega^-_2$ be equal to $\Omega^-$ for $\theta=\delta=0$. Since $T_4$ is a diffeomorphism, i.e., $b_{11}b_{22}-b_{12}b_{21}\neq 0$, one has $B_i\neq 0$ and $\omega^-_1\neq \omega^-_2$. In this notation, the map $T_4 |_{\Pi_{bot}}$ takes the form
\begin{equation}\label{eq:T4_bot}
T_4 |_{\Pi_{bot}}: \;
\begin{aligned}
\bar{x}_1&=B_1 \cos(\omega_1^-) \, x+B_1 \sin(\omega_1^-)\, y +O_2,\\
\bar{y}_1&=B_2 \cos(\omega_2^-) \, x+B_2 \sin(\omega_2^-)\, y +O_2.
\end{aligned}
\end{equation}
Since we introduce the symmetric local coordinates in a neighborhood of the point $O^+$, the map $T_4 |_{\Pi_{top}}$ has the form
\begin{equation}\label{eq:T4_top}
T_4 |_{\Pi_{top}}: \;
\begin{aligned}
\bar{x}_1&=B_2 \cos(\omega_2^+) \, x+B_2 \sin(\omega_2^+)\, y +O_2,\\
\bar{y}_1&=B_1 \cos(\omega_1^+) \, x+B_1 \sin(\omega_1^+)\, y +O_2,
\end{aligned}
\end{equation}
where $\omega^+_2=\Omega^+$ for $\theta=\delta=0$. Due to the symmetry, we can choose the coefficients $\omega_i^{\pm}$ so that $|\omega_i^--\omega_i^+|=\pi/2$.

For sufficiently small $\delta<0$ the separatrices $\Gamma^-_1$ and $\Gamma^-_2$, after passing through the cross-section $\Pi_{cyl}$, intersect the cross-sections $\Pi_{bot}$ and $\Pi^-$. So, points on the cross-section $\Pi^-$ with sufficiently small $y_1\neq 0$ return to this cross-section, and one can define the Poincar\'e map $T:\Pi^-\to\Pi^-$, which is a composition of the maps $T_4\circ T_3\circ T_2\circ T_1$, see Fig. \ref{Poincare_map}a.

If $\delta>0$,  then trajectories, starting at the cross-section $\Pi^+$ sufficiently close to the line $y_1=0$, intersect the cross-section $\Pi^{top}$ and, then, arrive to $\Pi^+$, i.e., the map $T=T_4\circ T_3\circ T_2\circ T_1$ acts from $\Pi^-$ to $\Pi^+$ in this case. So we consider the half-map $\hat{T}=S\circ T$, which again takes the cross-section $\Pi^-$ into itself, see Fig. \ref{Poincare_map}b.

\begin{figure}[h]
    \centering
    \includegraphics[width=0.7\linewidth]{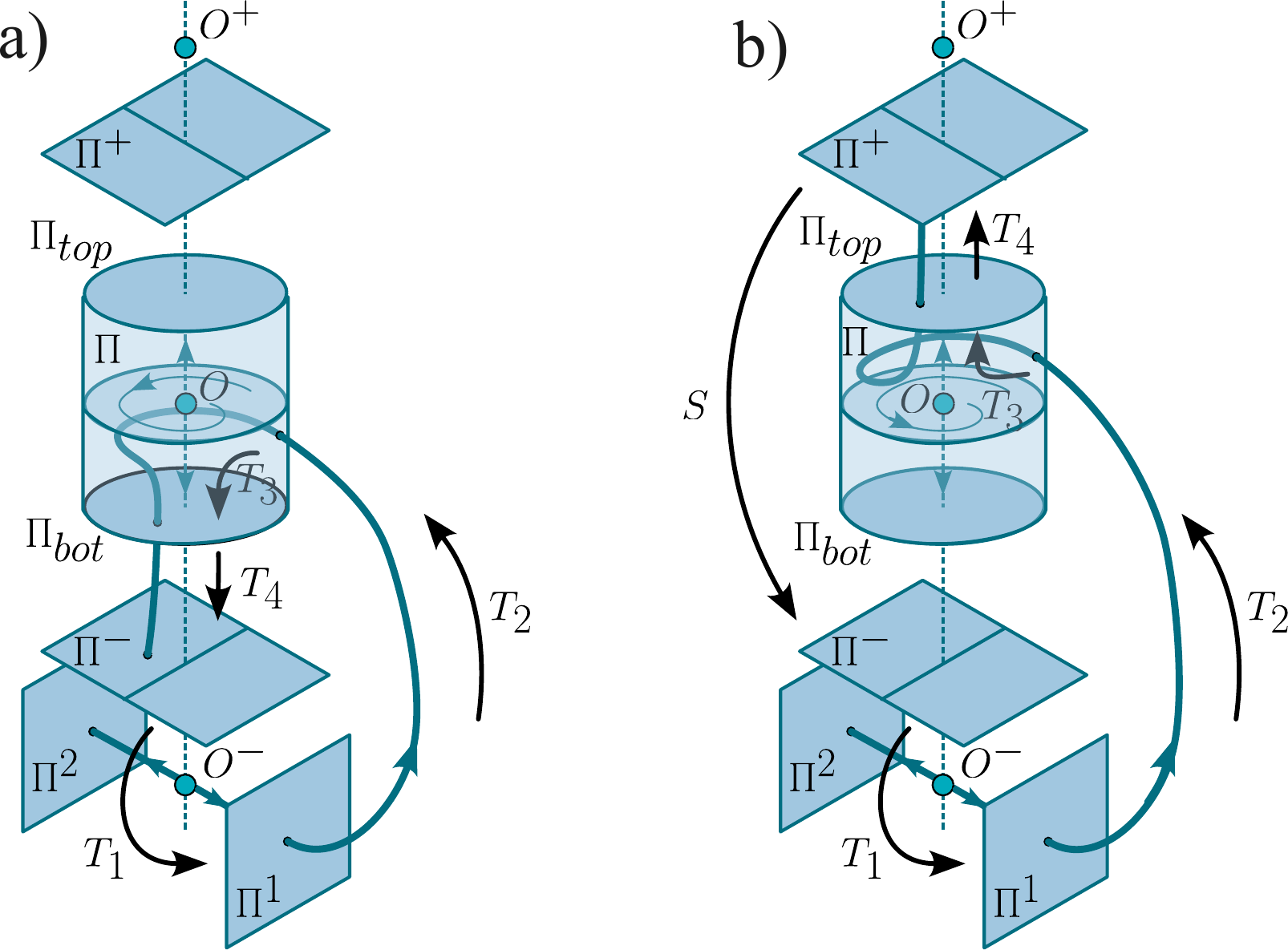}
    \caption{Construction of a) the Poincar\'e map  $T$ in the case $\delta<0$ (the Lorenz attractor case), and b) the half-map $\hat{T}$ in the case $\delta>0$ (the Sim\'o attractor case).}
    \label{Poincare_map}
\end{figure}

\begin{proposition}\label{th:Poincare_map}
Under assumptions B1-B5,  in the $(\theta,\delta)$-plane there are two sequences of regions $\tilde{\mathcal A}^{-}_k$ and $\tilde{\mathcal A}^{+}_k$ such that the following holds:

{\rm (1)} The regions $\tilde{\mathcal A}^{-}_k$ lie in the open quadrants $\{ \delta<0$, $(\Omega^--\pi k)\, \theta<0\}$ (we exclude $k=0$ if $\Omega^{-}=0$) and are given by
\begin{equation}\label{eq:domain_equation_A-}
\tilde{\mathcal A}^{-}_k:\;
\delta=-\exp \left(-\frac{\hat{\nu} |\Omega^--\pi k|}{|\theta|}+\Psi^-_k(\theta)+ p\, \hat{\Psi}^-_k(\theta)\right),
\end{equation}
where $\Psi^-_k(\theta)$ and $\hat{\Psi}^-_k(\theta)\neq 0$ are  uniformly bounded, and $p$ runs an interval between an arbitrarily small $p_1>0$ and an arbitrarily large $p_2>0$.

{\rm (2)} Similarly, the region $\tilde{\mathcal A}^{+}_k$  lies in the open quadrants $\{ \delta>0$, $(\Omega^+-\pi k)\, \theta <0\}$ (we exclude $k=0$ if $\Omega^{-}=0$) and given by
\begin{equation}\label{eq:domain_equation_A+}
\tilde{\mathcal A}^{+}_k:\;
\delta=\exp \left(-\frac{\hat{\nu} |\Omega^+-\pi k|}{|\theta|} +\Psi^+_k(\theta)+ p\, \hat{\Psi}^+_k(\theta)\right),
\end{equation}
where $\Psi^+_k(\theta)$ and $\hat{\Psi}^+_k(\theta)\neq 0$ are  uniformly bounded, and $p$ runs an interval between an arbitrarily small $p_1>0$ and an arbitrarily large $p_2>0$.

{\rm (3)} For any $(\theta,\delta)\in \tilde{A}^{\pm}_k$ there exists a linear rescaling of the coordinates in the cross-section $\Pi^-$ which brings the map $T$ for $\delta<0$ and the map $\hat{T}$ for $\delta>0$ to the form
\begin{equation}\label{eq:T}
\begin{aligned}
\bar{X}&=(1+\hat{c}\, |Y|^\nu+|\delta|^{s}\, \eta_1(X,Y))\cdot {\rm sign}(Y),\\
\bar{Y}&=\sigma \, (-1+c\, |Y|^{\nu}+|\delta|^{s}\, \eta_2(X,Y))\cdot {\rm sign}(Y),
\end{aligned}
\end{equation}
where $\sigma=(-1)^{k}\, {\rm sign}(a_{22}\, (\Omega^{\pm}-\pi k))\, {\rm sign}(\delta)$, the functions $\eta_{1,2}$ satisfy the following estimates
\begin{equation}\label{eq:eta12_estimates}
|\eta_{1,2}|=o(Y^{\tilde{\nu}}), \ \ \ \left|\frac{\partial \eta_{1,2}}{\partial X}\right|=o(|Y|^{\tilde{\nu}}), \ \ \ \left|\frac{\partial \eta_{1,2}}{\partial Y}\right|=o(Y^{\tilde{\nu}-1}),
\end{equation}
and $s$ is a positive constant such that $0<s<\min \left(\nu/\hat{\nu}-1, (\tilde{\nu}-\nu)/\hat{\nu}\right)$.

{\rm (4)} The coefficient $c$ is a positive smooth function of the parameters $\theta$ and $\delta$ such that for a fixed $p\in [p_1,p_2]$
$$
\lim_{(\theta,\delta)\to (0,0)} c=p^{\nu-1}.
$$
The coefficient $\hat{c}=o(|\delta|^s)$ is a smooth function of $\theta$ and $\delta$.
\end{proposition}

\begin{proof}
For definiteness, we consider the case $\delta<0$. In the case $\delta>0$, the calculations are absolutely similar with the replacement of $\omega_{1,2}^-$ by $\omega_{1,2}^+$. Since the Poincar\'e map $T$ is symmetric with respect to $x_1\to -x_1$, $y_1\to -y_1$ and a linear change of coordinates preserves this symmetry, it is enough to derive equation \eqref{eq:T} for positive $Y$ only.

By \eqref{eq:T1_map} and \eqref{eq:T2_map}, for positive $y_1$ the composition $T_2\circ T_1:\Pi^-\backslash \{y_1=0\}\to\Pi_{cyl}$ can be written as
\begin{equation}\label{eq:T1T2_map}
T_2\circ T_1|_{y_1>0}: \;
\begin{aligned}
\bar{\phi}&=a_{12} \, y_1^{\nu}+o(y_1^{\tilde{\nu}}),\\
\bar{z}&=\delta+a_{22} \, y_1^{\nu}+o(y_1^{\tilde{\nu}}),
\end{aligned}
\end{equation}
If $\delta<0$ and $y_1$ is small enough, then $\bar{z}<0$. Therefore, we only consider the map $T_3$ for negative $z$. By
\eqref{eq:T3_map} and \eqref{eq:T4_bot}, for $z<0$ the composition of the maps $T_3$ and $T_4$ is equal to
\begin{equation}\label{eq:T3T4_delta_negative}
T_4\circ T_3 |_{z<0}: \;
\begin{aligned}
\bar{x}_1&=B_1 e^{-\tau}\, \sin \left(\phi+\theta \tau +\omega^-_1\right)+O(e^{-2\tau}),\\
\bar{y}_1&=B_2 e^{-\tau}\, \sin \left(\phi+\theta \tau +\omega^-_2\right)+O(e^{-2\tau}),
\end{aligned}
\end{equation}
where
\begin{equation}\label{eq:tau_def}
\tau=-\frac{\ln|z|}{\hat{\nu}}.
\end{equation}
For simplicity of notation, we denote the right-hand side of equations \eqref{eq:T3T4_delta_negative} as
\begin{equation}\label{eq:h_definition}
\begin{aligned}
h_{1,2}(\tau,\phi)&=B_{1,2} e^{-\tau}\, \sin \left(\Phi_{1,2}\right)+O(e^{-2\tau}),\\
\end{aligned}
\end{equation}
where
\begin{equation}\label{eq:Phi_definition}
\Phi_{1,2}=\phi+\theta \tau +\omega^-_{1,2}.
\end{equation}
We also denote
\begin{equation}\label{eq:tau0_definition}
\tau_0=-\frac{\ln|\delta|}{\hat{\nu}}, \ \ \ \
\Phi_0=\theta \tau_0+\omega^-_2.
\end{equation}

Recall that $\omega_2^-$ is a function of $(\theta, \delta)$ such that $\omega_2^-(0,0)=\Omega^-$. Let us define the functions $\Psi_k^-$, $\hat{\Psi}_k^-$ in formula \eqref{eq:domain_equation_A-} such that
$$
\Psi_k^-(\theta)=\frac{\hat{\nu}}{\theta}\, \bigl(\omega_2^-(\theta,\delta_0(\theta))-\Omega^- \bigr),
$$
where $\delta_0(\theta)$ is obtained from \eqref{eq:domain_equation_A-} at $p=0$, i.e.,
$$
\delta_0(\theta)=-\exp \left(-\frac{\hat{\nu} |\Omega^--\pi k|}{|\theta|}+\Psi^-_k(\theta)\right),
$$
and
$$
\hat{\Psi}_k^-=\frac{\hat{\nu}}{\theta}\, {\rm sign}(a_{22}\, \theta)\, {\rm sign}(\delta) \,
\left(B_2^{\nu} \bigr(\delta_0(\theta)\bigl)^{\nu/\hat{\nu}-1}\, \frac{|a_{22} \theta|}{ \hat{\nu}}\right)^{1/(1-\nu)}.
$$
Then in the region $\tilde{\mathcal A}^{-}_k$ we have the following relation\footnote{We use the notation $\xi_1\sim \xi_2$ if $\xi_1/\xi_2\to 1$ as $\theta,\delta\to 0$.}:
\begin{equation}\label{eq:Domain_equation}
\Phi_0-\pi k\sim p\cdot {\rm sign}(a_{22}\, \theta)\, {\rm sign}(\delta) \,\left(B_2^{\nu} e^{-\tau_0(\nu-\hat{\nu})}\, \frac{|a_{22} \theta|}{\hat{\nu}}\right)^{1/(1-\nu)}
\end{equation}
or
\begin{equation}\label{eq:sin_value}
\begin{aligned}
\sin(\Phi_0)\sim -p\cdot \sigma\,\left(B_2^{\nu} e^{-\tau_0(\nu-\hat{\nu})}\, \frac{|a_{22} \theta|}{\hat{\nu}}\right)^{1/(1-\nu)},
\end{aligned}
\end{equation}
where
$$\sigma=(-1)^{k+1}\, {\rm sign}(a_{22}\, \theta)\, {\rm sign}(\delta)=(-1)^{k}\, {\rm sign}(a_{22}\, (\Omega^--\pi k))\, {\rm sign}(\delta),$$
and the parameter $p>0$ runs over a compact interval $[p_1,p_2]$. It also follows from formula \eqref{eq:Domain_equation} that
\begin{equation}\label{eq:theta_equiv}
\frac{\theta}{\omega_2^--\pi k} \sim -\frac{1}{\tau_0}=\frac{\hat{\nu}}{\ln |\delta|}.
\end{equation}
Note that if $\Omega^-\neq 0$, then $\omega_2^--\pi k$ is uniformly bounded away from $0$ for all $k\in \mathbb{Z}$ and all sufficiently small $\theta,\delta$. If $\Omega^-=0$, then we exclude $k=0$, and for the rest of $k\in \mathbb{Z}$, the value of $\omega_2^--\pi k$ is uniformly bounded away from $0$.

We rescale the coordinates in the cross-sections $\Pi^-$ and $\Pi_{cyl}$ by the following transformation
\begin{equation}\label{rescaling}
x_1=\varepsilon_{ss}\, X, \ \ y_1=\varepsilon\, Y, \ \
z=|\delta|\, Z, \ \ \phi=\varepsilon^{\nu} \varphi,
\end{equation}
where small $\varepsilon>0$ and $\varepsilon_{ss}$ are given by (see \eqref{eq:h_definition})
\begin{equation}\label{eq:eps_equation}
\begin{aligned}
&\varepsilon_{ss} = h_1(\tau_0,0) = B_1 e^{-\tau_0}\, \sin (\Phi_0+\omega^-_1-\omega^-_2)+O(e^{-2\tau_0}),\\
&\varepsilon = |h_2(\tau_0,0)| = |B_2 e^{-\tau_0}\, \sin (\Phi_0)+O(e^{-2\tau_0})|.
\end{aligned}
\end{equation}
Since $\omega_1^-\neq \omega_2^-$, it follows from \eqref{eq:sin_value} that
\begin{equation}\label{eq:order_of_tilde_eps}
\varepsilon_{ss}\sim (-1)^k B_1 e^{-\tau_0} \sin(\omega_1^--\omega_2^-).
\end{equation}
It also follows from \eqref{eq:sin_value} and assumption A3 that  $\sin(\Phi_0)\gg e^{-\tau_0}$. Hence
\begin{equation}\label{eq:sign_of_h_2}
h_2(\tau_0,0)\sim B_2 e^{-\tau_0}\, \sin(\Phi_0),  \ \ \  \ {\rm sign}(h_2(\tau_0,0))=-\sigma,
\end{equation}
and
\begin{equation}\label{eq:order_of_eps}
\varepsilon^{1-\nu}\sim |B_2 e^{-\tau_0}\, \sin (\Phi_0)|^{1-\nu}\sim p^{1-\nu}\cdot  B_2 e^{-(1-\hat{\nu})\tau_0}\, \frac{|a_{22}\theta|}{\hat{\nu}}\sim p^{1-\nu}\cdot B_2 \frac{e^{-(1-\hat{\nu})\tau_0}}{\tau_0}\, \frac{|a_{22}(\omega_2^--\pi k)|}{\hat{\nu}}.
\end{equation}

After scaling \eqref{rescaling}, the map $T_2\circ T_1$ given by \eqref{eq:T1T2_map} takes the form
\begin{equation}\label{eq:T1T2plus_rescaled}
T_2\circ T_1|_{Y>0}: \;
\begin{aligned}
\bar{\phi}&=a_{12}\, Y^{\nu}+ \varepsilon^{\tilde{\nu}-\nu}\, o(Y)^{\tilde{\nu}},\\
\bar{Z}&={\rm sign}(\delta)+\frac{\varepsilon^{\nu}}{|\delta|} \left(a_{22}\, Y^{\nu}+\varepsilon^{\tilde{\nu}-\nu}\, o(Y)^{\tilde{\nu}}\right),
\end{aligned}
\end{equation}
and the map $T_4\circ T_3$ given by \eqref{eq:T3T4_delta_negative} takes the form
\begin{equation}\label{eq:T3T4_rescaled}
T_4\circ T_3 |_{Z<0}:\,
\begin{aligned}
\varepsilon_{ss}\, \bar{X}&= h_1\bigl(\tau_0-\ln|Z|/\hat{\nu},\varepsilon^{\nu} \varphi\bigr),\\
\varepsilon\, \bar{Y}&=h_2\bigl(\tau_0-\ln|Z|/\hat{\nu},\varepsilon^{\nu} \varphi\bigl).
\end{aligned}
\end{equation}
Consider a subdomain of $\Pi^-$ given by $0<Y<1$. Note that by \eqref{eq:order_of_eps} we have
\begin{equation}\label{eq:eps_delta_relation}
\frac{\varepsilon^{\nu}}{|\delta|}=\varepsilon^{\nu} e^{\hat{\nu} \tau_0}\sim  const\,\cdot \left(\frac{e^{-(\nu-\hat{\nu})\tau_0}}{\tau_0^{\nu}}\right)^{\frac{1}{1-\nu}}=O(e^{-(\nu-\hat{\nu})\tau_0})=O(|\delta|^{\nu/\hat{\nu}-1}),
\end{equation}
so $\varepsilon^{\nu}$ is much smaller than $|\delta|$. In particular, for $0<Y<1$ and sufficiently small $\varepsilon$, the coordinate $\bar{Z}$ in \eqref{eq:T1T2plus_rescaled} is close to ${\rm sign}(\delta)=-1$. So for such $Y$ the map $T$ is the composition of \eqref{eq:T1T2plus_rescaled} and \eqref{eq:T3T4_rescaled}. Let us show that this composition can be brought to form \eqref{eq:T}.

Denote
\begin{equation}\label{eq:hat_c_def}
\hat{c}=\frac{1}{\varepsilon_{ss}}\, \left(\frac{\partial h_1(\tau_0,0)}{\partial \tau}\, \left(-\frac{\varepsilon^{\nu}}{\delta} \frac{a_{22}}{\nu}\right)+\frac{\partial h_1(\tau_0,0)}{\partial \phi}\, \varepsilon^{\nu} a_{12} \right),
\end{equation}
\begin{equation}\label{eq:c_def}
c=\frac{\sigma}{\varepsilon}\, \left(\frac{\partial h_2(\tau_0,0)}{\partial \tau}\,\left(-\frac{\varepsilon^{\nu}}{\delta} \frac{a_{22}}{\nu}\right)+\frac{\partial h_2(\tau_0,0)}{\partial \phi}\,  \varepsilon^{\nu} a_{12} \right),
\end{equation}
and let
\begin{equation}\label{eq:tilde_eta_def}
\begin{aligned}
\tilde{\eta}_{1}(X,Y)&=\frac{1}{\varepsilon_{ss}}\, h_1\bigl(\tau_0-\ln|\bar{Z}|/\hat{\nu},\varepsilon^{\nu}\bar{\varphi}\bigr)-\hat{c}\, |Y|^{\nu}-1,\\
\tilde{\eta}_{2}(X,Y)&=\frac{\sigma}{\varepsilon}\, h_2\bigl(\tau_0-\ln|\bar{Z}|/\hat{\nu},\varepsilon^{\nu}\bar{\varphi}\bigr)-c\, |Y|^{\nu}+1,
\end{aligned}
\end{equation}
where $\bar{Z}, \bar{\varphi}$ are given by \eqref{eq:T1T2plus_rescaled} and $h_{1,2}$ are given by \eqref{eq:h_definition}. It follows from
\eqref{eq:T1T2plus_rescaled} and \eqref{eq:T3T4_rescaled} that in this notation the map $T$ is given by
$$
T:\,
\begin{aligned}
\bar{X}&=(1+\hat{c}\, |Y|^\nu+ \tilde{\eta}_1(X,Y))\cdot {\rm sign}(Y),\\
\bar{Y}&=\sigma \, (-1+c\, |Y|^{\nu}+\tilde{\eta}_2(X,Y))\cdot {\rm sign}(Y).
\end{aligned}
$$
This map has the same form as map \eqref{eq:T}. Therefore, to prove the proposition, we need to establish asymptotics for $c$ and $\hat{c}$, and estimate the functions $\tilde{\eta}_{1,2}$.

To do this, note that the derivatives of $h_{1,2}$ are equal to
\begin{equation}\label{eq:derivative_of_h}
\begin{aligned}
\frac{\partial h_{1,2}(\tau,\phi)}{\partial \tau}&=B_{1,2} e^{-\tau} \bigl(-\sin (\Phi_{1,2})+\theta \cos (\Phi_{1,2}) \bigr)+O(e^{-2\tau})=O(e^{-\tau}),\\
\frac{\partial h_{1,2}(\tau,\phi)}{\partial \phi}&=B_{1,2} e^{-\tau} \cos (\Phi_{1,2}) +O(e^{-2\tau})=O(e^{-\tau}).
\end{aligned}
\end{equation}
By \eqref{eq:sin_value} and \eqref{eq:theta_equiv}, we also have $\sin(\Phi_0) \ll \theta$. Therefore
\begin{equation}\label{eq:h2_tau_deriv}
\frac{\partial h_{2}(\tau_0,0)}{\partial \tau}\sim B_{2} e^{-\tau_0}\, (-1)^k  \theta.
\end{equation}
It follows from \eqref{eq:order_of_tilde_eps} and \eqref{eq:derivative_of_h}  that
$$
\frac{1}{\varepsilon_{ss}}\,\frac{\partial h_{1}(\tau_0,0)}{\partial (\tau,\phi)}=O(1),
$$
and hence, by \eqref{eq:hat_c_def}, $\hat{c}=O(\varepsilon^\nu/\delta)$. By \eqref{eq:theta_equiv}, we have $\theta\gg \delta$, and, therefore, by \eqref{eq:order_of_eps}, \eqref{eq:c_def} and \eqref{eq:h2_tau_deriv},
$$
c\sim \frac{\sigma}{\varepsilon} \cdot  \frac{\partial h_{2}(\tau_0,0)}{\partial \tau} \cdot \left(-\frac{\varepsilon^{\nu}}{\delta} \frac{a_{22}}{\hat{\nu}}\right)\sim B_2 e^{-(1-\hat{\nu})\tau_0}\, \frac{|a_{22} \theta|}{\hat{\nu}}\, \varepsilon^{\nu-1}\sim p^{\nu-1}.
$$
This is in an agreement with item (4) of the proposition.

Next, we estimate the functions $\tilde{\eta}_{1,2}(X,Y)$.
By \eqref{eq:T1T2plus_rescaled}, we have
\begin{equation}\label{eq:tau_tau_0_difference}
-\frac{\ln |\bar{Z}|}{\hat{\nu}}=-\frac{\varepsilon^{\nu}}{\delta} \frac{a_{22}}{\hat{\nu}}\, |Y|^{\nu}+\max \left(\frac{\varepsilon^{2\nu}}{|\delta|^2}, \frac{\varepsilon^{\tilde{\nu}}}{|\delta|}\right)\,  o(|Y|^{\tilde{\nu}}).
\end{equation}
It follows from \eqref{eq:tau_tau_0_difference} that
$$
\frac{\partial h_{1,2}(\tau_0-\ln|\bar{Z}|/\hat{\nu},\varepsilon^{\nu}\bar{\varphi})}{\partial (\tau,\phi)}=\frac{\partial h_{1,2}(\tau_0,0)}{\partial (\tau,\phi)}\, \left(1+\frac{\varepsilon^{\nu}}{\delta}\, O(|Y|^{\nu})\right)
$$
Therefore, by \eqref{eq:T1T2plus_rescaled} and \eqref{eq:tau_tau_0_difference},
\begin{multline}\label{eq:Y_deriv_result}
\frac{\partial h_{1,2}\bigl(\tau_0-\ln|\bar{Z}|/\hat{\nu},\varepsilon^{\nu}\bar{\varphi}\bigr)}{\partial Y}=\\
=\frac{\partial h_{1,2}\bigl(\tau_0-\ln|\bar{Z}|/\hat{\nu},\varepsilon^{\nu}\bar{\varphi}\bigr)}{\partial \tau}\, \frac{\partial (-\ln|\bar{Z}|/\hat{\nu})}{\partial Y}
+\frac{\partial h_{1,2}\bigl(\tau_0-\ln|\bar{Z}|/\hat{\nu},\varepsilon^{\nu}\bar{\varphi}\bigr)}{\partial \phi}\, \frac{\partial (\varepsilon^{\nu}\bar{\varphi})}{\partial Y}=\\
=\left(-\frac{\partial h_{1,2}(\tau_0,0)}{\partial \tau}\, \frac{\varepsilon^{\nu}}{\delta} \, \frac{a_{22}}{\nu}+
\frac{\partial h_{1,2}(\tau_0,0)}{\partial \phi}\, \varepsilon^{\nu} a_{12}\right) \, \frac{\partial |Y|^{\nu}}{\partial Y}+e^{-\tau_0}\, \max \left(\frac{\varepsilon^{2\nu}}{|\delta|^{2}},\frac{\varepsilon^{\tilde{\nu}}}{|\delta|} \right)\, o(|Y|^{\tilde{\nu}-1})
\end{multline}
and
\begin{multline}\label{eq:X_deriv_result}
\frac{\partial h_{1,2}(\tau_0-\ln|\bar{Z}|/\hat{\nu},\varepsilon^{\nu}\bar{\varphi})}{\partial X}=\\
=\frac{\partial h_{1,2}\bigl(\tau_0-\ln|\bar{Z}|/\hat{\nu},\varepsilon^{\nu}\bar{\varphi}\bigr)}{\partial \tau}\, \frac{\partial (-\ln|\bar{Z}|/\hat{\nu})}{\partial X}
+\frac{\partial h_{1,2}\bigl(\tau_0-\ln|\bar{Z}|/\hat{\nu},\varepsilon^{\nu}\bar{\varphi}\bigr)}{\partial \phi}\, \frac{\partial (\varepsilon^{\nu}\bar{\varphi})}{\partial X}=\\
=e^{-\tau_0}\, \max \left(\frac{\varepsilon^{2\nu}}{|\delta|^2},\frac{\varepsilon^{\tilde{\nu}}}{|\delta|} \right)\, o(|Y|^{\tilde{\nu}}).
\end{multline}
Comparing \eqref{eq:Y_deriv_result}, \eqref{eq:X_deriv_result}  with  \eqref{eq:hat_c_def}, \eqref{eq:c_def} and \eqref{eq:tilde_eta_def},
we obtain
$$
\tilde{\eta}_1(X,Y)=\frac{e^{-\tau_0}}{\varepsilon_{ss}}\, \max \left(\frac{\varepsilon^{2\nu}}{|\delta|^2},\frac{\varepsilon^{\tilde{\nu}}}{|\delta|} \right)\, o(|Y|^{\tilde{\nu}}),
$$
$$
\tilde{\eta}_2(X,Y)=\frac{e^{-\tau_0}}{\varepsilon}\, \max \left(\frac{\varepsilon^{2\nu}}{|\delta|^2},\frac{\varepsilon^{\tilde{\nu}}}{|\delta|} \right)\, o(|Y|^{\tilde{\nu}}).
$$
Then, we have, by  \eqref{eq:order_of_eps},
$$
\frac{e^{-\tau_0}}{\varepsilon}\, \max \left(\frac{\varepsilon^{2\nu}}{|\delta|^2},\frac{\varepsilon^{\tilde{\nu}}}{|\delta|} \right)=\frac{e^{-(1-\hat{\nu})\tau_0}}{\varepsilon^{1-\nu}}\, \max \left(\frac{\varepsilon^{\nu}}{|\delta|},\varepsilon^{\tilde{\nu}-\nu} \right)\sim \frac{\rm const}{\theta}\, \max \left(\frac{\varepsilon^{\nu}}{|\delta|},\varepsilon^{\tilde{\nu}-\nu}\right),
$$
and, by \eqref{eq:theta_equiv} and \eqref{eq:eps_delta_relation},
$$
\frac{\rm 1}{\theta}\, \max \left(\frac{\varepsilon^{\nu}}{|\delta|},\varepsilon^{\tilde{\nu}-\nu}\right)= \ln|\delta|\cdot O \left(|\delta|^{\nu/\hat{\nu}-1}+|\delta|^{(\tilde{\nu}-\nu)/\hat{\nu}}\right)=o(\delta^{s}),
$$
where $0<s<\min(\nu/\hat{\nu}-1,(\tilde{\nu}-\nu)/\hat{\nu})$. Similarly,
$$
\frac{e^{-\tau_0}}{\varepsilon_{ss}}\, \max \left(\frac{\varepsilon^{2\nu}}{|\delta|^2},\frac{\varepsilon^{\tilde{\nu}}}{|\delta|} \right)=\frac{\varepsilon}{\varepsilon_{ss}}\, \frac{e^{-\tau_0}}{\varepsilon}\, \max \left(\frac{\varepsilon^{2\nu}}{|\delta|^2},\frac{\varepsilon^{\tilde{\nu}}}{|\delta|} \right)=o(\delta^{s}).
$$
Thus, the functions $\delta^{-s} \tilde{\eta}_{1,2}(X,Y)$ satisfy conditions \eqref{eq:eta12_estimates} of item 3 of the proposition. This completes the proof.
\end{proof}

To prove the existence of pseudohyperbolic attractors (Lorenz attractors in the regions $\mathcal A_k^-$ and Sim\'o attractors in the regions $\mathcal A_k^+$), we need to show that map \eqref{eq:T} is singular hyperbolic for an interval of values of $p$. We choose a compact interval $[p_1,p_2]$  that strictly  covers the interval $c(p)\in [1,2]$. For $c< 2$ and sufficiently small $\delta$, the open subdomain of $\Pi^-$ given by the inequality $|Y|<1$ is forward invariant with respect to the map. We need to establish the existence of invariant cone fields in this subdomain. The derivative of $T$ is equal to
$$
DT=\begin{pmatrix}
   \frac{\partial \bar{X}}{\partial X} & \frac{\partial \bar{X}}{\partial Y} \\
    \frac{\partial \bar{Y}}{\partial X} & \frac{\partial \bar{Y}}{\partial Y}
\end{pmatrix}=
\begin{pmatrix}
    |\delta|^{s} o(|Y|^{\tilde{\nu}}) & |\delta|^{s} \, O(|Y|^{\nu-1}) \\
    |\delta|^{s} o(|Y|^{\tilde{\nu}}) & c\nu\, |Y|^{\nu-1}+ |\delta|^{s} \, o(|Y|^{\tilde{\nu}-1})
\end{pmatrix}.
$$
It is easy to check that there are cone fields of the form
$$
C^{ss}: \, \{\, o(|\delta|^s) |Y|^{\tilde{\nu}-\nu+1}\,  |dX|\geq |dY|\, \}, \ \ \ C^u: \, \{\, |dX|\leq O(|\delta|^s) \,|dY|\, \},
$$
which are invariant with respect to $DT$. The derivative $DT$ contracts every vector of $C^{ss}$ by the factor $O(|\delta|^s) |Y|^{\tilde{\nu}}\ll 1$. The vectors in $C^{u}$ are multiplied by the factor $(c\nu+o(|\delta|^s))|Y|^{\nu}$. If $c\nu>1$, then, for sufficiently small $\delta$, the derivative $DT$ is uniformly expanding in $C^u$ and the map $T$ is singular hyperbolic. This proves Theorem \ref{th6} for $\nu>1/2$.

In general case, the question of the singular hyperbolicity of the map $T$ becomes non-trivial because we do not have an immediate expansion in $C^u$. We give a complete solution to this problem in \cite{preprint2}. There we describe the exact region in the $(c,\nu)$-plane where the map $T$ has a singular hyperbolic attractor ($DT$ becomes expanding in $C^u$ for an appropriate choice of variables).  In particular, results of  \cite{preprint2} imply

\begin{theorem}\label{th7}
There is an open connected region of values of $c$ and $\nu$ such that map \eqref{eq:T} has a singular hyperbolic attractor for sufficiently small $\delta$. This region for all $\nu\in (0,1)$ contains an interval of $c$ from $[1,2]$.
\end{theorem}

\noindent Proposition \ref{th:Poincare_map} and Theorem \ref{th7} give Theorem \ref{th6}.

Now we finish the proof of Theorem~\ref{th3}.

\begin{proposition}
Let $A<0$ in system \eqref{FlowNormalFormRescaled} and let $\rho=h(\mu,\omega)$ be the curve of the four-winged heteroclinic bifurcation given by \eqref{eq:het_curve}. Then, for all sufficiently small fixed $\mu>0$, system \eqref{FlowNormalFormRescaled} satisfies assumptions B1-B5 with $\theta$ and $\delta$ being equal to $\omega/\rho$ and $h(\mu,\omega)-\rho$, respectively.
\end{proposition}

\begin{proof} Assumptions B1, B2 are obviously fulfilled and assumption B4 follows from Theorem \ref{th5}. Let us check assumption B3. It needs to be checked at $\omega=0$ and $\rho=h(\mu,0)$. Therefore, we have $(\rho,\mu)$ close to $(1/2,0)$. Then, for the equilibrium $O$ we have
$$
\hat{\nu}=\frac{1}{\rho}\, \left(\frac{\sqrt{b_1}}{|a_0|}\, \sqrt{\mu}+O(\mu)\right)>0
$$
and for the equilibria $O^+$ and $O^-$ we have
$$
\nu= \frac{\frac{2\sqrt{b_1}}{|a_0|}\, \sqrt{\mu}+O(\mu)}{1-\rho+C\sqrt{\mu}+O(\mu)}>0.
$$
Obviously, $\nu/\hat{\nu}\to 2$ as $(\rho,\mu)\to (1/2,0)$, which implies $0<\hat{\nu}<\nu$. The inequality $\nu<(\hat{\nu}+1)/2<1$ is true because $\nu$ and $\hat{\nu}$ tend to $0$ as $\mu\to 0$. So, assumption B3 holds.

It remains to check B5. For $\omega=\mu=0$ the extended unstable manifold of $O^-$ coincides with the $(y,z)$-plane and the stable manifold of $O$ coincides with the $(x,y)$-plane. So, these manifolds intersect transversely. Since these manifolds depend smoothly on the parameters, the transversal intersection is preserved for close values of the parameters, and assumption B5 is also true.
\end{proof}

This completes the proof of Theorem~\ref{th3}.
\bigskip

Different values of $k$ in Theorem~\ref{th3} correspond to the unstable separatrices making different numbers of rotations along the $z$-axis before hitting the cross-section. We illustrate this in Figure~\ref{fig_3Simo}, where we plot portraits of attractors found at the point $s_1$ (Fig.~\ref{fig_3Simo}a), $s_2$ (Fig.~\ref{fig_3Simo}b), and $s_3$ (Fig.~\ref{fig_3Simo}c) marked in Fig.~\ref{bifurcation_diagram}. It follows from \eqref{eq:domain_A-_equation} and \eqref{eq:domain_A+_equation} that the largest regions $\mathcal{A}_k^{\pm}$ correspond to $k=0$. This is in agreement with the numerics, see Fig. \ref{bifurcation_diagram}. However, the case $k=0$ might require additional consideration if $\Omega^-$ or $\Omega^+$ vanish. To avoid this problem, we prove that  for the heteroclinic bifurcation in system \eqref{FlowNormalFormRescaled}, the values of $\Omega^-$ and $\Omega^+$ are not equal to $0$ for sufficiently small $\mu>0$.

\begin{figure}[h]
    \centering
    \includegraphics[width=1.0\linewidth]{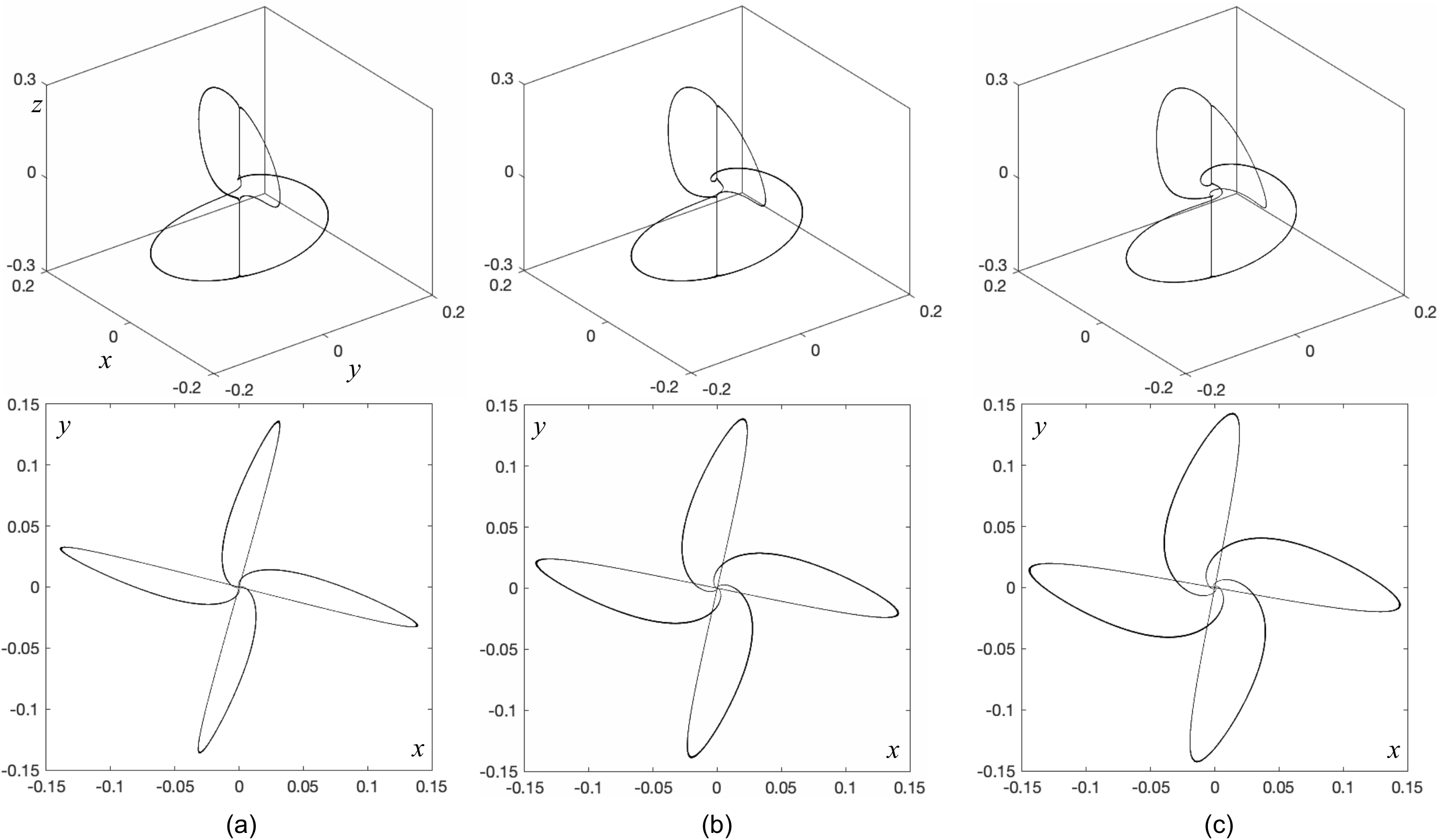}
    \caption{Different Simo angels existing in the regions $\mathcal {A}^+_{1,2,3}$ (at points $s_{1,2,3}$, see Fig.~\ref{bifurcation_diagram}) of system \eqref{FlowNormalFormNumeric} at $\mu = 0.02$: a) $(\beta, \gamma) = (0.042, 0.098)$, b) $(\beta, \gamma) = (0.06, 0.089)$; c) $(\beta, \gamma) = (0.07, 0.0813)$.}
    \label{fig_3Simo}
\end{figure}

Indeed, for $\omega=\mu=0$ the stable manifolds of $O^-$ and $O^+$ are contained in the $(x,z)$-plane and the $(y,z)$-plane, respectively. The tangent subspace $V(z)$ to the strong stable foliation coincides with the line $z=const, y=0$ for $-1<z<0$ and coincides with the line $z=const, x=0$ for $0<z<1$. So, the limit subspaces $V^-$ and $V^+$ are equal to $z=y=0$ and $z=x=0$, respectively. For $\omega=\mu=0$ and $\rho=1/2$, the unstable separatrices of $O^-$ come to $O$ along the $y$-axis. So for such values of the parameters, $\Omega^-=\pi/2$ and $\Omega^+=0$. The value of $\Omega^-$ depends smoothly on the parameters $\rho$ and $\mu$ for fixed $\omega=0$, therefore the value of $\Omega^-$  is close to $\pi/2$. Let us check that $\Omega^+$ is not equal to $0$ for $\mu>0$.

\begin{proposition} Let $A<0$ in system \eqref{FlowNormalFormRescaled}. Then the angle $\Omega^+$ is given by
$$\Omega^+=-\frac{\sqrt{\mu}}{6A} \, \left(\frac{{\rm Im}(a_1)} {|a_0|}+\frac{|A|\, {\rm Im}(a_2)}{ |a_0|\, \sqrt{|b_1|}}+{\rm Im}(a_3/a_0^2)\,|a_0| \right)+O(\mu).
$$
\end{proposition}

\begin{proof}
In the proof of Theorem \ref{th5} we obtained the equation for the unstable separatrix of the equilibrium $O^+$ at $\omega=0$ and $\rho=h(\mu,0)$, see formula \eqref{eq:unstable_manifold}. We will calculate the direction along which this separatrix comes to $O$ for $\mu>0$. Let the separatix be given by functions $(x(t),y(t),z(t))$. Since $W^s(O)$  is tangent to the $(x,y)$-plane, it is enough to calculate $y(t)/x(t)$ as $t\to+\infty$. We have
$$
\frac{d}{d\sqrt{\mu}} \left(\frac{y(t)}{x(t)}\right) \Big|_{\mu=0}=\frac{\frac{dy(t)}{d\sqrt{\mu}}\, x(t)-y(t)\, \frac{dx(t)}{d\sqrt{\mu}}}{x^2(t)}\, \Big|_{\mu=0}=\frac{\hat{Y}(t)}{x_0(t)},
$$
where $\hat{Y}(t)=\frac{dy(t)}{d\sqrt{\mu}}|_{\mu=0}$ and
$$
x_0(t)=\frac{e^{t/2}}{\sqrt{2|A|}\, (1+e^t)}.
$$
By differentiating  \eqref{nonhomogeneous_system} with respect to $\sqrt{\mu}$, we obtain that $\hat{Y}(t)$ satisfies
$$
\frac{d\hat{Y}}{dt}=-(z_0(t)+1/2)\, \hat{Y}+\hat{g}_2(x_0,z_0),\\
$$
where
$$
\hat{g}_2(x_0,z_0)=\frac{{\rm Im}(a_1)}{|a_0|}\,  x_0^3(t)+ \frac{{\rm Im}(a_2)}{ |a_0|\, \sqrt{|b_1|}} \, z_0^2(t) x_0(t)+{\rm Im}(a_3/a_0^2)\,|a_0| \, x_0^3(t)
$$
and
$$
z_0(t)=\frac{1}{1+e^{t}}.
$$
One can check that
$$
\hat{Y}(t)\stackrel{t\to+\infty}{\sim} -\frac{\sqrt{2}\,e^{-t/2}}{12A\sqrt{|A|}} \, \left(\frac{{\rm Im}(a_1)} {|a_0|}+\frac{|A|\, {\rm Im}(a_2)}{ |a_0|\, \sqrt{|b_1|}}+{\rm Im}(a_3/a_0^2)\,|a_0| \right).
$$
Thus,
$$
\lim_{t\to+\infty} \frac{\hat{Y}(t)}{x_0(t)}=-\frac{1}{6A} \, \left(\frac{{\rm Im}(a_1)} {|a_0|}+\frac{|A|\, {\rm Im}(a_2)}{ |a_0|\, \sqrt{|b_1|}}+{\rm Im}(a_3/a_0^2)\,|a_0| \right).
$$
Since $O$ is a dicritical node on its stable manifold for $\omega=0$, the direction in which the separatrix comes to $O$ depends smoothly on the parameters, and hence this direction is given by
$$
\lim_{t\to+\infty}\frac{y(t)}{x(t)}=-\frac{\sqrt{\mu}}{6A} \, \left(\frac{{\rm Im}(a_1)} {|a_0|}+\frac{|A|\, {\rm Im}(a_2)}{ |a_0|\, \sqrt{|b_1|}}+{\rm Im}(a_3/a_0^2)\,|a_0| \right)+O(\mu).
$$

Thus, to prove the proposition, it is enough to show that $V^+$ coincides with the $y$-axis up to $O(\mu)$-terms. We find the stable manifolds $W^s(O^+)$ in the form
$$
x=H(y,z)=\sqrt{\mu}\, y\, (\eta(z;\sqrt{\mu})+O(y)).
$$
It contains the $z$-axis and coincides with the $(y,z)$-plane for $\mu=0$. We need to check that the limit of $\eta(z;\sqrt{\mu})$ as $z\to 0+$ is $O(\sqrt{\mu})$. The surface $x=H(y,z)$ must satisfy the invariance condition
\begin{equation}\label{eq:invariance_condition}
\dot{x}=\frac{\partial H}{\partial y} \dot{y}+\frac{\partial H}{\partial z} \dot{z}=\sqrt{\mu}\, (\eta(z)+O(y))\, \dot{y}+\sqrt{\mu}\,y\, (\eta'(z)+O(y))\, \dot{z} .
\end{equation}
As in \eqref{eq:CD_def}, we denote
$$
C+iD= \frac{a_2}{|a_0|\, \sqrt{|b_1|}}.
$$
By \eqref{FlowNormalFormRescaled}, \eqref{eq:high_order_terms}, for $x=H(y,z)$ we have
\begin{equation}\label{eq:system_on_Ws}
\begin{aligned}
\dot{x}&= (-\rho +z)\,y\, \sqrt{\mu} \,\eta(z)+\sqrt{\mu}\, z^2 y\, \bigl(C\sqrt{\mu}\, \eta(z)-D+O(\mu)\bigr)+\sqrt{\mu}\, O(y^2)\\
\dot{y}&=(-\rho -z)\,y +\sqrt{\mu}\, z^2 y\, \bigl(C+D \sqrt{\mu}\, \eta(z) +O(\mu)\bigr)+\sqrt{\mu}\, O(y^2)\\
\dot{z}&=\sqrt{\mu}\left(\frac{\sqrt{b_1}}{|a_0|}+O(\sqrt{\mu})\right)\, (z-z^3) +O(y^2)
\end{aligned}
\end{equation}
Substituting \eqref{eq:system_on_Ws} in \eqref{eq:invariance_condition} and omitting the $O(y^2)$-terms, we obtain the equation for $\eta(z)$
\begin{equation}\label{eq:eta_equation}
2\eta(z)-z\, D-D\mu \,z \, \eta(z)=\eta'(z)\,\sqrt{\mu}\left(\frac{\sqrt{b_1}}{|a_0|}+O(\sqrt{\mu})\right)(1-z^2).
\end{equation}
For $\mu=0$ this equation takes the form
$$
2\eta(z)-zD=0
$$
or
$$
\eta(z)=\frac{z D}{2}.
$$
One can check that equation \eqref{eq:eta_equation} has a solution bounded for $z\in[0,1]$ and smoothly depending on $\mu$. This solution corresponds to the sought function $\eta(z)$ in the formula for the stable manifold, and has the form
$$
\eta(z)=\frac{z D}{2}+O(\sqrt{\mu}),
$$
Thus,
$$
\eta(0)=O(\sqrt{\mu}).
$$
This completes the proof.

\end{proof}

{\bf Acknowledgments.}
The authors are grateful to Kirill Zaichikov, for the help with computation of Lyapunov diagrams.

The work was supported by the Leverhulme Trust grant RPG-2021-072, the RSF grant No. 19-71-10048 (Sections~\ref{Pseudohyperbolic_attractors} and \ref{sec:Existence_of_LA}), the RSF grant No. 23-71-30008 (Section~\ref{sec4}), and by the Basic Research Program at HSE (Section~\ref{Shilnikov_criterion}).

\end{document}